\documentclass[reqno, 12pt]{amsproc}

\usepackage[utf8]{inputenc}
\usepackage{amsmath}
\usepackage[margin=.925in]{geometry}
\usepackage{mathtools}
\usepackage{amssymb}
\usepackage{extarrows}
\usepackage{amsthm}
\usepackage{amsmath,mathabx}
\usepackage{mathtools}
\usepackage{cases, mathdots}
\usepackage{booktabs}
\usepackage{blkarray}
\usepackage[english]{babel}
\usepackage{booktabs}
\usepackage{enumitem}
\usepackage{comment}

\usepackage[linesnumbered,ruled,vlined]{algorithm2e}
\SetKwInput{KwInput}{Input}                
\SetKwInput{KwOutput}{Output}              
\usepackage[hidelinks,colorlinks=true,linkcolor=blue,citecolor=blue]{hyperref}
\usepackage[all,cmtip]{xy}
\usepackage{pgfplots,mathpazo}
\pgfplotsset{compat=1.17}
\usepackage{mathrsfs}
\usetikzlibrary{arrows,decorations.pathreplacing,patterns}

\newcommand{\smat}[1]{\left[\begin{smallmatrix}#1\end{smallmatrix}\right]}

\theoremstyle{plain}
\newtheorem{theorem}{Theorem}[section]
\newtheorem*{theorem*}{Theorem}
\newtheorem{prop}[theorem]{Proposition}
\newtheorem{lemma}[theorem]{Lemma}
\newtheorem{cor}[theorem]{Corollary}
\theoremstyle{remark}
\newtheorem{rem}[theorem]{Remark}
\newtheorem{ex}[theorem]{Example}

\theoremstyle{definition}
\newtheorem{definition}[theorem]{Definition}

\renewcommand{\mathcal}{\mathscr}
\newcommand{\bbN}{\mathbb{N}}
\newcommand{\bbZ}{\mathbb{Z}}
\newcommand{\bbQ}{\mathbb{Q}}
\newcommand{\bbR}{\mathbb{R}}
\newcommand{\bbC}{\mathbb{C}}

\newcommand{\set}[1]{\left\{#1\right\}}
\DeclareMathOperator{\udim}{\underline\dim}
\newcommand{\K}{\mathbb{F}} 
\newcommand{\Int}{\text{\bf I}}
\newcommand{\Barc}{{\text{\bf Bar}}}

\newcommand{\inds}{\mathcal{I}}

\newcommand{\Rep}{\text{\bf Rep}}
\newcommand{\Rect}{\Rep_{\text{\rm rec}}}
\newcommand{\bHN}{\text{\bf HN}}
\newcommand{\bR}{\mathbf{R}}

\newcommand{\HN}[2]{\bHN_{#2}^\bullet(#1)}
\newcommand{\subrep}[1]{\langle {#1} \rangle}

\newcommand{\typeA}{\mathbb A}
\newcommand{\HNtype}[2]{\mathbf{T}[{#2};{#1}]}
\DeclareMathOperator{\Ind}{\text{\rm Ind}}

\newcommand{\eq}{\text{\rm eq}}
\newcommand{\nf}{\text{\rm nf}}
\newcommand{\supp}{\text{\rm supp}}

\newcommand{\sfn}{s}
\newcommand{\tfn}{t}
\newcommand{\VS}{{Q_0}}
\newcommand{\ES}{{Q_1}}

\tikzstyle{redstyle}=[red,fill, fill opacity=0.5,pattern=north west lines, pattern color=red ]
\newcommand{\figureunderlyinggrid}[3]{%
\begin{tikzpicture}[scale=#1]
\foreach \x in {0,...,7}{
\node (\x+) at (\the\numexpr\x+0.5,-0.5){};
}
\foreach \x in {0,...,7}{
\node (\x-) at (\x+0.5,-1.5){};
}
#2
\foreach \x in {0,...,7}{
\draw[gray!40!black!60,->] (\x+) -- (\x-);
}
\foreach \x in {0,...,6}{
\draw[gray!40!black!60,->] (\x+) -- (\the\numexpr(\x+1)+);
\draw[gray!40!black!60,->] (\x-) -- (\the\numexpr(\x+1)-);
}
  #3
\end{tikzpicture}%
}

\title{Harder-Narasimhan Filtrations of Persistence Modules}
\author{Marc Fersztand, Emile Jacquard, Vidit Nanda and Ulrike Tillmann}

\begin{document}

\maketitle
\begin{abstract}
The Harder-Narasimhan type of a quiver representation is a discrete invariant parameterised by a real-valued function (called a central charge) defined on the vertices of the quiver. In this paper, we investigate the strength and limitations of Harder-Narasimhan types for several families of quiver representations which arise in the study of persistence modules. We introduce the skyscraper invariant, which amalgamates the HN types along central charges supported at single vertices, and generalise the rank invariant from multiparameter persistence modules to arbitrary quiver representations. Our four main results are as follows: (1) we show that the skyscraper invariant is strictly finer than the rank invariant in full generality, (2) we characterise the set of complete central charges for zigzag (and hence, ordinary) persistence modules, (3) we extend the preceding characterisation to rectangle-decomposable multiparameter persistence modules of arbitrary dimension; and finally, (4) we show that although no single central charge is complete for nestfree ladder persistence modules, a finite set of central charges is complete.
 \end{abstract}

\section{Introduction} \label{sec: intro}

More than sixty years ago, the notion of {\em semistability} was introduced by Mumford \cite{mumfordicm, mumford} to construct well-behaved quotient spaces for the actions of reductive groups on algebraic varieties. Subsequently, Harder and Narasimhan described a stratification of the moduli space of finite rank vector bundles over a complex curve for the purpose of computing its cohomology groups \cite{hnfilt}. The top stratum consists of semistable bundles, and every bundle lying in a lower stratum admits a canonical filtration of finite length whose associated graded components are semistable with strictly decreasing slopes (i.e., ratios of degree to rank). This {\em Harder-Narasimhan filtration} continues to play a vital role in moduli problems involving vector bundles and coherent sheaves \cite{atiyah-bott, huylehn}; its existence and uniqueness have been established more generally for objects of certain abelian categories by Rudakov \cite{rudakov} to triangulated categories by Bridgeland \cite{bridgeland2007stability} and to modular lattices by Haiden et al \cite{haiden2020semistability}. The work of King \cite{king} and Reineke \cite{Reineke_2003} has extended the Harder-Narasimhan formalism to categories of quiver representations. 

 Our efforts here stem from the desire to use Harder-Narasimhan theory to build isomorphism invariants for multiparameter persistence modules. In general, such modules are representations of wild type quivers, so there is no hope of obtaining a complete\footnote{An invariant is {\em complete} if two persistence modules are isomorphic whenever their invariants are equal.
} discrete invariant. Nevertheless, the quest for discriminative invariants has been a central theme within topological data analysis. The earliest work in this direction was by  Carlsson and Zomorodian \cite{multi}, who proposed the {\em rank invariant}. Subsequent efforts to study multiparameter persistence modules have involved a plethora of tools sourced from diverse locales --- these include sheaf theory \cite{macpat, ks}, commutative and homological algebra \cite{miller, host, bettis}, lattice theory \cite{patel, saecular, botnanoudot} and beyond \cite{rivet}.

\subsection*{Outline and summary of results}
Fix a finite quiver $Q$ and let $\Rep(Q)$ denote the category of finite-dimensional representations of $Q$ valued in vector spaces. By a {\em central charge} on $Q$, we mean any real-valued function $\alpha$ defined on the set $\VS$ of vertices\footnote{This is the imaginary part of an abelian group homomorphism $K(\Rep(Q)) \to \bbC$; see Section \ref{ssec:hnabelian}.}. Consider a nontrivial representation $V \in \Rep(Q)$, which assigns vector spaces $V_x$ to vertices $x \in \VS$. The $\alpha$-slope of $V$ is the ratio
\[
\mu_\alpha(V) = \frac{\sum_{x \in \VS}{\alpha(x) \cdot \dim V_x}}{\sum_{x \in \VS} \dim V_x},
\]
and $V$ is said to be $\alpha$-semistable whenever the inequality $\mu_\alpha(V) \geq \mu_\alpha(V')$ holds for every nontrivial subrepresentation $V' \subset V$. 

Once we fix a central charge $\alpha$ on $Q$, every nonzero representation $V$ admits a unique {\em Harder-Narasimhan} filtration \cite{Reineke_2003, harada} of finite length 
\[
0 = \bHN^0_\alpha(V) \subsetneq \bHN^1_\alpha(V) \subsetneq \cdots \subsetneq \bHN^{n-1}_\alpha(V) \subsetneq \bHN^n_\alpha(V) = V
\]
whose successive quotients $S^i := \bHN^i_\alpha(V)/\bHN^{i-1}_\alpha(V)$ are $\alpha$-semistable and satisfy $\mu_\alpha(S^i) > \mu_\alpha(S^{i-1})$ for all $i$. The Harder-Narasimhan filtration thus provides a canonical method for building arbitrary quiver representations out of the $\alpha$-semistable ones. The {\bf HN type} of $V$ along $\alpha$ is the $n$-tuple of functions
\[
\HNtype{\alpha}{V} = \left(\udim_{S^1},\udim_{S^2} \ldots, \udim_{S^n}\right),
\]
where $\udim_{S^i}:\VS \to \bbN$ assigns the natural number $\dim S^i_x$ to each vertex $x \in \VS$.

An important role in our paper is played by the {\em spanning} subrepresentation of $V$ at a vertex $x$ --- this is defined up to isomorphism as the smallest subrepresentation $\subrep{V_x} \subset V$ containing $V_x$. The function $\rho_V:\VS \times \VS \to \bbN$ that sends each $(x,y)$ to the dimension of $ \subrep{V_x}_y$ vastly generalises the rank invariant of Carlsson and Zomorodian. Consider, for each vertex $x$, the central charge $\delta_x:\VS \to \bbR$ which maps $x$ to $1$ and all other vertices to $0$. The skyscraper invariant $\delta_V$ is defined on $\Rep(Q)$ as the collection of HN types $\HNtype{\delta_x}{V}$ indexed by the vertices of $Q$. 

Our first main result is presented in Section \ref{sec:skyscraper}, and may be stated as follows. 

\begin{theorem*} [\textbf{A}] The skyscraper invariant is finer than the rank invariant on $\Rep(Q)$ for any finite quiver $Q$.
\end{theorem*}
\noindent The full statement may be found in Theorem \ref{thm:rankfromHN}, where we provide a precise formula for recovering $\rho_V(x,y)$ in terms of $\HNtype{\delta_x}{V}$. 
In forthcoming work, we will address the computability and stability (with respect to the interleaving distance \cite{lesnickthesis}) of the skyscraper invariant.

We then specialise in Section \ref{section:pathquiver} to the setting of zigzag persistence modules --- these are finite-dimensional representations of type $\mathbb A$ quivers of arbitrary (but finite) length $\ell \geq 0$:
\[
\xymatrixcolsep{.6in}
\xymatrix{
x_0 \ar@{-}[r]^{e_1} & x_1 \ar@{-}[r]^{e_2} & \cdots \ar@{-}[r]^{e_{\ell-1}} & x_{\ell-1} \ar@{-}[r]^{e_{\ell}} & x_{\ell}.
}
\]
Here each edge $e_i$ points either forward $x_{i-1} \to x_i$ or backward $x_{i-1} \leftarrow x_i$, and ordinary persistence modules correspond to the equioriented case where all of the $e_i$ point forward. It is well known from Gabriel's theorem \cite{gabriel} that any representation $V$ of such a quiver decomposes uniquely as a direct sum of {\em interval representations}, which are supported on sub-intervals $[a,b] \subset [0,\ell]$. The intervals which appear with nonzero multiplicity comprise the {\em barcode} of $V$ --- see \cite{zigzag_pers}. We exploit knowledge of these indecomposables as well as some recent work of Kinser \cite{totalstability} to prove the following result, which is Theorem \ref{theorem: mainAl} below.

\begin{theorem*} [\textbf{B}] There is a classification of all complete central charges on the category of zigzag persistence modules; in the special case of ordinary persistence modules, a central charge $\alpha:\VS \to \bbR$ is complete iff $\alpha(x_i) > \alpha(x_{i+1})$ holds for all $i$.    
\end{theorem*}

The natural $d$-dimensional analogues of ordinary persistence modules, for $d > 1$, are certain representations of {\em grid quivers} whose vertices are parameterised by integer points in the product $[0,\ell_1] \times [0,\ell_2] \times \cdots \times [0,\ell_d]$,  with each $\ell_i \geq 1$. The $d=2$ case is illustrated below:
\[
\xymatrixcolsep{.8in}
\xymatrixrowsep{.25in}
\xymatrix{
x_{0,\ell_2} \ar@{->}[r] & x_{1,\ell_2} \ar@{->}[r]_{} & \cdots \ar@{->}[r]_{} & x_{\ell_1,\ell_2} \\
\vphantom{\int^0}\smash[t]{\vdots} \ar@{->}[u] & \vphantom{\int^0}\smash[t]{\vdots} \ar@{->}[u]_{} & {\iddots}  &  \vphantom{\int^0}\smash[t]{\vdots}  \ar@{->}[u]_{}\\
x_{0,1} \ar@{->}[r]^{} \ar@{->}[u]_{} & x_{1,1} \ar@{->}[r]^{} \ar@{->}[u]_{} & \cdots \ar@{->}[r]^{}  &  x_{\ell_1,1} \ar@{->}[u]_{} \\
x_{0,0} \ar@{->}[r]^{} \ar@{->}[u]_{} & x_{1,0} \ar@{->}[r]^{} \ar@{->}[u]_{} & \cdots \ar@{->}[r]^{}  & x_{\ell_1,0} \ar@{->}[u]_{} 
}
\]
The representations of interest are those for which every rectangle in sight commutes as a diagram of vector spaces. Unlike the $d=1$ case, these quivers are of wild representation type; as mentioned above, much effort has been invested towards finding good discrete invariants for multiparameter persistence modules. A far more tractable subcategory is spanned by the {\em rectangle-decomposable} representations, which admit direct sum decompositions into the obvious $d$-dimensional analogue of interval modules. Here we establish the following result, consisting of Corollary \ref{cor:skyrank} and Theorem \ref{thm:complete_param_rect_equiv}.

\begin{theorem*} [\textbf{C}]
    The skyscraper invariant is strictly finer than the rank invariant on the category of multiparameter persistence modules. Moreover, for the subcategory of rectangle decomposable modules, there is a classification of complete central charges which lie outside of a hyperplane arrangement in the space of maps $\VS \to \bbR$.
\end{theorem*}

In Section \ref{sec:ladders}, we focus on to the special case of $2$-parameter persistence modules arising from representations of {\em ladder quivers} of length $\ell \geq 1$:
\[
\xymatrixcolsep{.6in}
\xymatrixrowsep{.3in}
\xymatrix{
x_0^+ \ar@{->}[r] \ar@{->}[d] & x_1^+ \ar@{->}[r] \ar@{->}[d] & \cdots \ar@{->}[r]  & x_{\ell-1}^+ \ar@{->}[r] \ar@{->}[d] & x_\ell^+ \ar@{->}[d] \\
x_0^- \ar@{->}[r] & x_1^- \ar@{->}[r] & \cdots \ar@{->}[r] & x_{\ell-1}^- \ar@{->}[r] & x_\ell^-
}
\]
Again, every rectangle is required to commute. These {\em ladder persistence modules} arise from morphisms of (ordinary) persistence modules. Ladder quivers are known \cite{escolar2016persistence, wildladder} to be of finite representation type only for $\ell \leq 3$. We restrict attention to the subcategory spanned by those representations whose top and bottom rows, when viewed as ordinary persistence modules, do not admit a pair of strictly nested intervals in their barcode decompositions\footnote{ Two intervals
$[i, j]$ and $[i', j']$  are strictly 
nested if  $i<i'$ and $j'<j$.}. It was shown in \cite[Section 5]{barcodebases} that this subcategory is representation-finite for all $\ell$. Here we obtain the following results (Proposition \ref{prop:laddercounterexample} and Theorem \ref{thm:mainladder}) for such {\em nestfree} ladder persistence modules.

\begin{theorem*}[\textbf{D}]
    There is no complete central charge on the category of nestfree ladder persistence modules of length $\ell \geq 4$; however, for all $\ell$ there exists a finite set $A = A(\ell)$ of central charges which is complete on this category.
\end{theorem*}

\noindent The set $A(\ell)$ is explicitly described in Definition \ref{def:ladderparams}, and its cardinality grows cubically with $\ell$; every constituent element of this set is an $\bbR$-linear combination of at most two skyscraper central charges.

\subsection*{Acknowledgments} The authors are members of the Centre for Topological Data Analysis, funded by the EPSRC grant EP/R018472/1. VN is partially supported by US AFOSR grant FA9550-22-1-0462 and thanks Fabian Haiden and Frances Kirwan for several helpful discussions. The authors are grateful to Veronica Phan for providing in \cite{moanswer} the key idea behind Lemma \ref{lmm:techn_flow}.

\section{HN types of quiver representations}\label{sec:hntypes} 

In this Section, we establish notation and recall pertinent aspects of quiver representations \cite{kirillov_quiver, schiffler2014quiver}, their Harder-Narasimhan filtrations, and the associated Harder-Narasimhan types \cite{Reineke_2003, harada}.

\subsection{Quiver representations} A quiver $Q$ consists of the following data: a set $\VS$ whose elements are called vertices, a set $\ES$ whose elements are called edges, and a pair of functions $\sfn,\tfn:\ES \to \VS$ called the source and target map respectively. One often depicts each edge $e \in \ES$ as an arrow $\sfn(e) \longrightarrow \tfn(e)$. A path in the quiver $Q = (\sfn,\tfn:\ES \to \VS)$ is any finite sequence of edges $p = (e_1,\ldots,e_k)$ satisfying $\tfn(e_i) = \sfn(e_{i+1})$ for all $i$, and $Q$ is called {\em acyclic} if it admits no such paths with $\sfn(e_1) = \tfn(e_k)$. In this paper, all quivers are assumed to have only finitely many vertices and edges, and moreover, we will primarily be interested in acyclic quivers.

We work over a field $\K$ which remains fixed throughout (and hence suppressed from the notation), so that all vector spaces encountered henceforth are defined over $\K$. We recall that a {\bf representation} $V$ of $Q$ constitutes assignments of vector spaces $V_x$ to vertices $x \in \VS$ and linear maps $V_e: V_{\sfn(e)} \to V_{\tfn(e)}$ to edges $e \in \ES$. All representations considered here are finite-dimensional in the sense that $\dim V_x < \infty$ holds for each vertex $x \in \VS$; the \textit{dimension vector} of any such $V$ is the function $\udim_V:\VS \to \mathbb{N}$ given by $x \mapsto \dim V_x$. The set $\Rep(Q)$ of all finite-dimensional representations of $Q$ is readily upgraded to a category as follows. A morphism $\phi:V \to W$ consists of linear maps $\set{\phi_x:V_x \to W_x \mid x \in \VS}$ which make the following diagram commute for each edge $e \in \ES$:
\[
\xymatrixcolsep{1in}
\xymatrixrowsep{.45in}
\xymatrix{
V_{\sfn(e)} \ar@{->}[r]^{\phi_{\sfn(e)}} \ar@{->}[d]_{V_e} & W_{\sfn(e)} \ar@{->}[d]^{W_e} \\
V_{\tfn(e)} \ar@{->}[r]_{\phi_{\tfn(e)}}  & W_{\tfn(e)} 
}
\] 
We call $\phi$ an (epi, iso, mono)-morphism if each $\phi_x$ is (sur, bi, in)-jective; $V$ is called a {\em subrepresentation} of $W$, denoted $V \subset W$, whenever there exists a monomorphism $\phi:V \to W$. The category $\Rep(Q)$ is known to be abelian, with composition, (co)kernels and (co)products being defined vertex-wise \cite[Section 1.3]{schiffler2014quiver}.

A representation of $Q$ is called {\bf indecomposable} if it does not admit any nontrivial direct sum decompositions in $\Rep(Q)$. It is known \cite[Theorem 1.2]{schiffler2014quiver} that for every finite-dimensional representation $V$ of a finite quiver $Q$, there exists a unique set $\Ind_Q(V)$ containing isomorphism classes of indecomposables  and a unique multiplicity function $d_V:\Ind_Q(V) \to \bbZ_{>0}$ for which there is an isomorphism
\begin{align}\label{eq:krullschmidt}
V \simeq \bigoplus_{I} I^{d_V(I)},
\end{align}
with $I$ ranging over $\Ind_Q(V)$. The seminal work of Gabriel \cite{gabriel} established that  the set of indecomposables in $\Rep(Q)$ with a fixed dimension vector is finite if and only if the undirected graph associated to $Q$ is a finite union of simply laced Dynkin diagrams.

\subsection{Harder-Narasimhan filtrations in abelian categories}\label{ssec:hnabelian}
The \textit{Grothendieck group} of an abelian category $\mathcal{C}$ is the abelian group $K(\mathcal C)$ freely generated by the isomorphism classes $[V]$ of  objects $V$ in $\mathcal{C}$ modulo a relation of the form $[V] = [U]+[W]$ for each short exact sequence $0\to U\to V\to W\to 0$ in $\mathcal{C}$. Let $(\bbC,+)$ denote the abelian group of complex numbers under addition; we recall that a \textbf{stability condition}\footnote {Such a homomorphism is sometimes called a {\it linear  stability function} \cite{totalstability} or a {\em central charge} \cite{bridgeland2007stability}.} 
\cite{bridgeland2007stability} on $\mathcal{C}$ is any group homomorphism 
\[
Z:K(\mathcal C)\to (\bbC,+)
\] that sends nonzero objects to the open half-plane $\{z\in \bbC\mid \mathrm{Re}(z)>0\}$. The $Z$-\textbf{slope} of an object $V \neq 0$ is the real number
\[
\mu_Z(V) :=\dfrac{\text{Im } Z(V)}{\text{Re }Z(V)}.
\]
We say that $V$ is $Z$-\textbf{semistable} if the inequality $\mu_Z(U) \leq \mu_Z(V)$ holds for every nonzero subobject $U \subset V$. If this inequality is strict for all subobjects $U \notin \set{0,V}$, then $V$ is said to be $Z$-\textbf{stable}.

The following result is a direct consequence of \cite[Proposition 2.4]{bridgeland2007stability}. We recall that an  abelian category is said to satisfy the Noetherian (respectively, Artinian) hypothesis if every sequence $\cdots \subset a_1 \subset a_0$ of subobjects (respectively, every sequence $b_0 \twoheadrightarrow b_1 \twoheadrightarrow \cdots$ of quotient objects)  eventually stabilises up to isomorphism.

\begin{theorem}\label{thm:HN_existence}
Let $\mathcal C$ be any abelian category which satisfies the Noetherian and Artinian hypotheses. Fix a stability condition $Z$ on $\mathcal C$. Every nonzero object $V$ of $\mathcal C$ admits a unique filtration $V^\bullet$ of finite length $n \geq 1$:
\begin{align} \label{equ:hnfiltration}
    0=V^0 \subsetneq V^1 \subsetneq  \dots \subsetneq V^n=V,
\end{align}
whose successive quotients $S^i := V^i/V^{i-1}$ are $Z$-semistable and have strictly decreasing slopes:
\[
\mu_Z(S^1) > \mu_Z(S^2) > \cdots > \mu_Z(S^n).
\] 
This $V^\bullet$ is called the \textbf{Harder-Narasimhan} (or \textbf{HN}) filtration of $V$ along $Z$.
\end{theorem}

Crucially, the category of finite dimensional representations of a finite quiver satisfies the hypotheses of the above result. It follows from uniqueness that the HN filtration of a $Z$-semistable object $V$ always has length one, i.e., $0 \subsetneq V$.

\subsection{HN types of quiver representations}\label{ssec:hnacyclic}
Let $Q=(\sfn,\tfn:\ES\to \VS)$ be a quiver which remains fixed throughout this subsection. The assignment $V \mapsto \udim_V$ defines a group homomorphism  from the Grothendieck group of $\Rep(Q)$ to the group of functions from $Q_0$ to $ \bbZ$:
\[
\udim: K(\Rep(Q)) \longrightarrow \bbZ^{Q_0}
\]
(For acyclic $Q$, this is an isomorphism \cite[Theorem 1.15]{kirillov_quiver}). We note that any stability condition $Z$ on $\Rep(Q)$ which factors through $\udim$ amounts to a choice of two maps $\alpha:\VS \to \bbR$ and $\beta:\VS \to \bbR_{>0}$. Explicitly, for nonzero $V \in \Rep(Q)$ we have
\[
Z(V) = \sum_{x \in \VS} \left(\beta(x) + \sqrt{-1} \cdot \alpha(x)\right) \cdot \dim V_x,
\]
and the corresponding slope is the ratio
\[
\mu_Z(V) = \frac{\sum_{x \in \VS} \alpha(x) \cdot \dim V_x}{\sum_{x \in \VS} \beta(x) \cdot \dim V_x}.
\]

In this paper, we will only work with stability conditions $Z$ which factor through $\udim$. Furthermore,
as in \cite{Reineke_2003, harada}, we further restrict attention to those stability conditions for which $\beta$ is the constant map sending all vertices to $1$. These are called {\em standard} stability conditions; and any such stability condition $Z$ depends on a single function $\alpha:\VS \to \bbR$ that we will henceforth call the {\bf central charge} of $Z$. In light of these simplifications, we will denote the slope of any nonzero $V$ by 
\begin{align}\label{eq:alphaslope}
\mu_\alpha(V) = \frac{\sum_{x \in \VS} \alpha(x) \cdot \dim V_x}{\sum_{x \in \VS} \dim V_x}.
\end{align} Similarly, the Harder-Narasimhan filtration of any nonzero $V \in \Rep(Q)$ along $Z$ is indicated by $\HN V \alpha$. The following {\em seesaw lemma} is stated in \cite[Lemma 2.2]{Reineke_2003}; we include a proof here for completeness. 
\begin{lemma} \label{lemma:sumstability}
Let $\alpha:\VS\to \bbR$ be a central charge for $Q$. Given any three nonzero objects which fit into a short exact sequence in $\Rep(Q)$:
\[
0 \to U \to V \to W \to 0,
\] 
their $\alpha$-slopes must satisfy one of the following chains of (in)equalities. Either
\begin{enumerate}
	\item $\mu_\alpha(U) > \mu_\alpha(V) > \mu_\alpha(W)$, or
	\item $\mu_\alpha(U) = \mu_\alpha(V) = \mu_\alpha(W)$, or 
	\item $\mu_\alpha(U) < \mu_\alpha(V) < \mu_\alpha(W)$.
\end{enumerate}
In Case (2), $V$ is $\alpha$-semistable if and only if both $U$ and $W$ are $\alpha$-semistable.
\end{lemma}
\begin{proof} 
Since $\arctan:\bbR \to (-\pi/2,\pi/2)$ is a strictly monotone increasing function, it suffices to establish the desired inequalities for the composite $\theta := \arctan \circ \mu_\alpha$ rather than for $\mu_\alpha$. By definition, the (standard) stability condition $Z$ induced by $\alpha$ is a group homomorphism $K(\Rep(Q)) \to (\bbC,+)$, so we have $Z(U) + Z(W) = Z(V)$. The desired results now follow from examining the parallelogram in $\bbC$ determined by the origin, $Z(U)$ and $Z(W)$ whose fourth point must be $Z(V)$. The angle $\theta(V)$ lies between the angles $\theta(U)$ and $\theta(W)$, with equality of all three angles occurring only in the degenerate case where $Z(U)$ is an $\bbR$-multiple of $Z(W)$. 

Let us now consider the case (2) where $U$, $V$ and $W$ share a common $\alpha$-slope $\mu$. If $U$ is not $\alpha$-semistable, then it admits a subrepresentation $U'$ with $\mu_\alpha(U') > \mu$; but any such $U'$ is automatically a subrepresentation of $V$ which violates its $\alpha$-semistability. Similarly, any quotient $W'$ of $W$ with $\mu_\alpha(W') < \mu$ violates the semistability of $V$. Thus, the semistability of $V$ forces semistability of both $U$ and $W$. Conversely, assume that $U$ and $W$ are $\alpha$-semistable and label the maps in the short exact sequence as $\iota:U \to V$ and $\pi:V \to W$. Given any subrepresentation $V' \subset V$, we have a short exact sequence
\[
    0\to \iota^{-1}(V')\to V'\to \pi(V')\to 0.
\]
 In the nontrivial case, $\iota^{-1}(V')$ and $\pi(V')$ are nonzero subrepresentations of $U$ and $W$ respectively, so by semistability both
  must have $\alpha$-slopes no larger than $\mu$. By the first part of this Lemma, we therefore obtain $\mu_\alpha(V') \leq \mu$, which confirms the $\alpha$-semistability of $V$.
\end{proof}

Here is an immediate (but important) consequence of the preceding result.

\begin{cor}\label{cor:directsumss}
If $U$ and $W$ are $\alpha$-semistable objects of $\Rep(Q)$ with the same $\alpha$-slope $\mu$, then their direct sum $U \oplus W$ is also $\alpha$-semistable with slope $\mu$.
\end{cor}

\begin{definition} \label{def:hntype}
The {\bf Harder-Narasimhan type}, or ({\bf HN type}) of a representation $V \neq 0$ in $\Rep(Q)$ along a central charge $\alpha:\VS \to \bbR$ is denoted $\HNtype{\alpha}{V}$ and defined as follows. Let $n$ be the length of the Harder-Narasimhan filtration $\HN{V}{\alpha}$, and let $S^i = \bHN^i_\alpha(V)/\bHN^{i-1}_\alpha(V)$ denote its successive quotients for $1 \leq i \leq n$. Then, 
\[
\HNtype{\alpha}{V} := \Big(\udim_{S^1}, \udim_{S^2}, \ldots, \udim_{S^n}\Big).
\]
\end{definition}

In the context of the preceding definition, it is often useful to view $\HNtype{\alpha}{V}$ as a function $\bbR \to \mathbb{N}^{\VS}$ in the following manner: 
\begin{align}\label{rem:hntypenotation}
\HNtype{\alpha}{V}(\lambda)=
\begin{cases}
\udim_{S^i}& \text{if }  \lambda=\mu_\alpha(S^i)  \textrm{ for some }  i,  \\
(0,0,\ldots,0) & \text{otherwise}.
\end{cases}
\end{align}
This function is well-defined since the successive quotients $S^i$ have strictly decreasing slopes by Theorem \ref{thm:HN_existence}. The uniqueness promised by this theorem further guarantees that $\HNtype{\alpha}{\bullet}$ is invariant under isomorphisms in $\Rep(Q)$. It is evident from Definition \ref{def:hntype} that this invariant is discrete, since it only produces finite sequences of (non-negative) integer values.  Moreover, this invariant is additive under direct sums; a proof of the following folklore result may be found in \cite[Proposition 2.5]{hnzigzag}.

\begin{prop} \label{prop:slopesum}
Let $V$ and $W$ be two representations of a quiver $Q$, and let $\alpha:\VS \to \bbR$ be a central charge. Adopting the notation of \eqref{rem:hntypenotation}, we have 
\[
\HNtype{\alpha}{V \oplus W}=\HNtype{\alpha}{V}+\HNtype{\alpha}{W}
\] as functions $\bbR \to \mathbb{N}^\VS$. 
\end{prop}

\subsection{Complete central charges}

The main theme of this paper is to measure the strength of the HN type as an invariant of certain quiver representations across various choices of central charge. The definition below corresponds to the best-case scenario.

\begin{definition}\label{def:complete}
Let $Q = (\sfn,\tfn:\ES \to \VS)$ be an acyclic quiver and $\mathcal{C}$ any subcategory of $\Rep(Q)$. A collection of central charges $A$ is said to be {\bf complete} on $\mathcal{C}$ if $\HNtype{\alpha}{V} = \HNtype{\alpha}{W}$ for all $\alpha \in A$ implies that $V$ and $W$ are isomorphic in $\mathcal{C}$.
\end{definition}

 If a collection of central charges $A$ is complete on all of $\Rep(Q)$, then we simplify terminology by saying that $A$ is complete for $Q$. For the simplest quivers, one hopes to find a single complete central charge; we will appeal to the following result frequently in our quest for such central charges.

\begin{lemma} \label{lemma:completenesscondition}
If $\alpha:\VS \to \bbR$ is a complete central charge for the acyclic quiver $Q $, then every indecomposable representation in $\Rep(Q)$ is $\alpha$-stable.
\end{lemma}
\begin{proof} Assume that $\alpha$ is a complete central charge and consider an indecomposable $I$ in $\Rep(Q)$. If $I$ is not $\alpha$-semistable, then its HN filtration
\[
0 \subsetneq \bHN^1_\alpha(I) \subsetneq \bHN^2_\alpha(I) \subsetneq \cdots \subsetneq \bHN^n_\alpha(I) = I
\]
has length $n > 1$. Abbreviating $I^i := \bHN^i_\alpha(I)$, in particular we have $I^1 \subsetneq I$. Now consider the filtration of $I^1 \oplus (I/I^1)$ given by:
\[
0 \subsetneq I^1 \subsetneq I^1 \oplus (I^2/I^1) \subsetneq \cdots \subsetneq I^1 \oplus (I^n/I^1).
\]
Since the successive quotients of this filtration are identical to those of $\HNtype{\alpha}{I}$, it follows (from uniqueness) that this new filtration is precisely $\HN{I^1\oplus(I/I^1)}{\alpha}$. Moreover, since the Harder-Narasimhan type depends only on these successive quotients, we have $\HNtype{\alpha}{I} = \HNtype{\alpha}{I^1\oplus(I/I^1)}$. But since $I$ is indecomposable, it can not be isomorphic to $I^1 \oplus (I/I^1)$ for $I^1 \neq I$, so the completeness of $\alpha$ forces $\alpha$-semistability of $I$. 

Given this $\alpha$-semistability, if $I$ is not $\alpha$-stable, then there exists a nonzero subrepresentation $J \subsetneq I$ with $\mu_\alpha(J)=\mu_\alpha(I)$. Using the exact sequence $0 \to J \to I \to I/J \to 0$ along with Lemma \ref{lemma:sumstability}, we know that both $J$ and $I/J$ are $\alpha$-semistable with slope $\mu_\alpha(I)$. Another appeal to the same Lemma establishes that the direct sum $J \oplus (I/J)$ is also $\alpha$-semistable with slope $\mu_\alpha(I)$. For dimension reasons, $\HNtype{\alpha}{I}$ equals $\HNtype{\alpha}{J \oplus (I/J)}$. But once again, since $I$ is indecomposable, it is not isomorphic to $J \oplus (I/J)$ for $J \neq I$. Thus, if $I$ is not $\alpha$-stable, then $\alpha$ is not a complete central charge for $Q$. 
\end{proof}

\section{The skyscraper and rank invariants}\label{sec:skyscraper}

We study the invariants  defined by  delta functions on  vertices of $Q$, the so-called skyscrapers, and show that they provide a finer invariant than the rank invariant which we define here for representations of any $Q$. 

\subsection{Skyscraper invariant}
Let $Q = (\sfn,\tfn:\ES \to \VS)$ be an arbitrary (i.e., finite, but not necessarily acyclic) quiver. In the absence of specific knowledge regarding the structure of $Q$ or its indecomposable representations, it is not immediately obvious how one might identify interesting classes of central charges for $Q$ \`a la Theorem \ref{theorem:totalstability}. Among the simplest nontrivial central charges which may be defined on any quiver are the ones supported on a single vertex.

\begin{definition}\label{def:skyscraper-weight}
The {\bf skyscraper} central charge at a vertex $x \in \VS$ is the map $\delta_x:\VS \to \bbR$ given by 
\[
\delta_x(y)=
\begin{cases} 1 & \text{if } y=x, \\
0 & \text{otherwise}.
\end{cases}
\]
And the {\bf skyscraper invariant} $\delta_\bullet$ on $\Rep(Q)$ assigns to each representation $V$ the collection of HN types $\delta_V =\set{\HNtype{\delta_x}{V} \mid x \in \VS}$ along skyscraper central charges at all of the vertices.
\end{definition}

\begin{definition}\label{def:spanrep}
Let $S \subset \VS$ be a nonempty subset of vertices. The {\bf spanning subrepresentation} of $V$ at $S$, denoted $\subrep{V_S}$, is the intersection of all subrepresentations $W \subset V$ for which $W_x$ is isomorphic to $V_x$ whenever $x$ lies in $S$. 
\end{definition}

We will simply write 
$\subrep{V_x}$ when $S$ is the singleton $\{x\}$. Spanning representations at singletons determine the HN filtrations along skyscraper central charges.

\begin{prop} \label{prop:filtrationsurjectivity} Given a vertex $x$ of $Q$, let 
$ 0=V^{0} \subsetneq \cdots \subsetneq V^{n}=V$ be the HN filtration of $V$ along $\delta_x$. If $j$ is the smallest index for which $V_x^j$ equals $V_x$, then:
\begin{enumerate}
    \item either $j=n$ or $j=n-1$, and
    \item for every $1 \leq k \leq j$, we have $V^k = \subrep{V^k_x}$.
\end{enumerate} 
\end{prop}
\begin{proof}
We note that the $\delta_x$-slope of a nonzero representation $W$ of $Q$ is given by
\begin{align}\label{eq:skyslope}
\mu_{\delta_x}(W) = \frac{\dim W_x}{\sum_{y \in \VS} \dim W_y},
\end{align}
which is evidently non-negative. Assuming that $V_x^j=V_x$ holds for some $j$ in $\set{0,\ldots,n}$, we have $V^k_x = V_x$ for all $k \geq j$, whence the successive quotients $S^k := V^k/V^{k-1}$ satisfy $S^k_x = 0$ for all $k > j$. By \eqref{eq:skyslope}, we obtain equalities of slopes: 
\[
0=\mu_{\delta_x}(S^{j+1}) = \mu_{\delta_x}(S^{j+2}) = \cdots = \mu_{\delta_x}(S^{n-1}) = \mu_{\delta_x}(S^n).
\]
Since these slopes are required to strictly decrease in the HN filtration, there are only two possible options. Either $j = n-1$, in which case only the last slope is $0$; or $j = n$, in which case all slopes are non-zero. Thus, we have established assertion (1).  We now prove assertion (2) by induction on $k \in \set{1,\ldots,j}$.  

{\bf Base case:} Since $V^1$ is $\delta_x$-semistable and $\subrep{V^1_x}$ is its subrepresentation, we must have $\mu_{\delta_x}(V^1) \geq \mu_{\delta_x}(\subrep{V^1_x})$, whence 
\[
\frac{\dim V^1_x}{\sum_{y \in \VS}\dim V^1_y} \geq \frac{\dim V^1_x}{\sum_{y \in \VS} \dim ~\subrep{V^1_x}_y}.
\]
If $\dim V^1_x = 0$ then there is nothing to check, so we assume that this dimension is nonzero. Since each $\subrep{V^1_x}_y$ is a subspace of the corresponding $V^1_y$ for  $y \in \VS$, we obtain $V^1=\subrep{V^1_x}$.

{\bf Inductive step:} Assume that $V^k=\subrep{V_x^k}$ holds for some $k < j$. Now $S^{k+1}$ is $\delta_x$-semistable by definition of the HN filtration; and by the argument which established assertion (1), it has a strictly positive $\delta_x$-slope. Thus, we have $\dim S^{k+1}_x > 0$, and applying the base case (to $S^{k+1}$ instead of $V^1$) yields $S^{k+1}=\subrep{S^{k+1}_x}$. Consequently, given any vertex $y \geq x$ and vector $\eta \in V_y^{k+1}$, there exists a vector $\xi \in V_x^{k+1}$ for which the difference $\eta' := V_{x \leq y}(\xi) - \eta$ lies in $V_y^k$. By the inductive hypothesis, this $\eta'$ must equal $V_{x \leq y}(\xi')$ for some $\xi' \in V^k_x$. Therefore, we have $V_{x \leq y}(\xi-\xi') = \eta$, whence $V^{k+1}$ equals $\subrep{V^{k+1}_x}$ as desired.
\end{proof}

\subsection{Rank invariant}

The following notion constitutes a substantial generalisation of the {\em rank invariant}, which was introduced in \cite{multi} for multi-parameter persistence modules.

\begin{definition}\label{def:rankinv}
The {\bf rank invariant} of $V \in \Rep(Q)$ is the map $\rho_V:\VS \times \VS \to \bbN$ given by:
\[
\rho_V(x,y) := \dim~ \subrep{V_x}_y.
\]
\end{definition}

\noindent It follows from the above definition that $\rho_\bullet$ is a discrete isomorphism invariant for $\Rep(Q)$. The following result gives a formula for the rank invariant in terms of the skyscraper invariant $\delta_\bullet$ --- in fact, we show that for any vertex $x$, the rank $\rho_V(x,y)$ can be recovered from the single Harder-Narasimhan type $\HNtype{\delta_x}{V}$. 

\begin{theorem} \label{thm:rankfromHN} Let $Q$ be a finite quiver. The skyscraper invariant is strictly more discriminative than the rank invariant on $\Rep(Q)$ in the following sense.
\begin{enumerate}
    \item Consider $V \in \Rep(Q)$ and any vertex $x$ in $\VS$. If $0=V^{0} \subsetneq \cdots \subsetneq V^{n}=V$ is the HN filtration of $V$ along $\delta_x$, then for any vertex $y \geq x$ of $Q$ we have
\[
\rho_V(x,y) = \sum_{k=1}^j \dim S^{k}_y,
\]
where $S^k := V^k/V^{k-1}$ and $j$ is the smallest index for which $V_x^j$ equals $V_x$.
   \item There exist two representations $W$ and $W'$ of the quiver
   \[
    \xymatrixcolsep{.75in}
    \xymatrixrowsep{.4in}
    
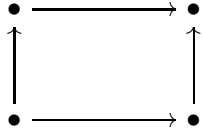
\xymatrix
    {
        \bullet \ar@{->}[r] & \bullet \\
        \bullet \ar@{->}[u] \ar@{->}[r] & \bullet \ar@{->}[u]   
    } 
   \]
   for which $\rho_W = \rho_{W'}$ whereas $\delta_W \neq \delta_{W'}$.
\end{enumerate}
\end{theorem}
\begin{proof}
By Proposition \ref{prop:filtrationsurjectivity}, we have $j \in \set{n,n-1}$ and $V^j=\subrep{V_x}$. So by Definition \ref{def:rankinv}, the value of $\rho_V(x,y)$ equals $\dim V_y^j$. Since the $S^\bullet$ are successive quotients of the HN filtration $V^\bullet$, we have
\[
\dim V^j_y = \sum_{k=1}^j \dim S^k_y,
\]
which establishes the first assertion. Turning now to the second assertion, let us consider the representations $W$ (left) and $W'$ (right) depicted below: 
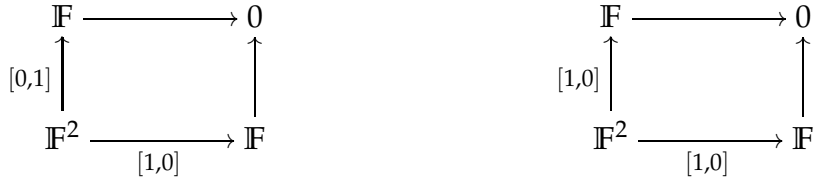
\[
\xymatrixcolsep{.75in}
\xymatrixrowsep{.4in}
\xymatrix{
\K \ar@{->}[r]  & 0 & & \K \ar@{->}[r]  & 0  \\
\K^2 \ar@{->}[r]_{[1,0]} \ar@{->}[u]^{[0,1]} & \K \ar@{->}[u]  & & \K^2 \ar@{->}[r]_{[1,0]} \ar@{->}[u]^{[1,0]} & \K \ar@{->}[u]
}
\]
Both evidently have the same rank invariant. Let $x$ be the $\leq$-minimal vertex of this quiver (i.e., the vertex on the bottom-left). By examining (the slopes of) sub-representations, one readily checks that $W$ is $\delta_x$-semistable, so that its HN filtration is just the trivial one $0 \subsetneq W$. On the other hand, $W'$ has a two-step HN filtration $0 \subset U \subset W'$, where $U$ is given by 
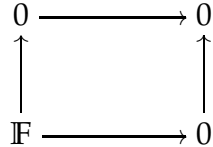
\[
\xymatrixcolsep{.75in}
\xymatrixrowsep{.4in}
\xymatrix{
0 \ar@{->}[r]  & 0 \\
\K \ar@{->}[r] \ar@{->}[u] & 0 \ar@{->}[u]
}  
\]
Since $W$ and $W'$ have different HN types along $\delta_x$, the skyscraper invariants $\delta_W$ and $\delta_{W'}$ are distinct as claimed above.
\end{proof}

\section{HN types of zigzag persistence modules} \label{section:pathquiver}

The goal of this section is to characterise complete central charges for type $\typeA_\ell$ quivers.

\subsection{Zigzag persistence modules}
Fix an integer $\ell \geq 0$. A quiver $Q$ is said to be of {\bf type $\typeA_\ell$} whenever its underlying undirected graph has the form
\[
\xymatrixcolsep{.6in}
\xymatrix{
x_0 \ar@{-}[r]^{e_1} & x_1 \ar@{-}[r]^{e_2} & \cdots \ar@{-}[r]^{e_{\ell-1}} & x_{\ell-1} \ar@{-}[r]^{e_{\ell}} & x_{\ell}.
}
\]
We describe the direction of edges via a boolean string $\tau$ of length $\ell$, called the {\em orientation} of $Q$: its $i$-th entry $\tau_i$ indicates whether $e_i$ points forward ($1$) from $e_{i-1}$ to $e_{i}$ or  backward ($0$) from $e_i$ to $e_{i-1}$. For instance, when $\ell = 3$, the sequence $\tau = 100$ implicates the following quiver:
\[
\xymatrixcolsep{.6in}
\xymatrix{
x_0 \ar@{->}[r]^{e_1} & x_1 & x_2 \ar@{->}[l]_{e_2} & x_3 \ar@{->}[l]_{e_3}.
}
\]
We say that $Q$ is {\em equioriented} if every $\tau_i$ equals $1$ (i.e., all edges point forward). Our goal here is to describe all complete central charges $\alpha$ for representations of type $\mathbb{A}_\ell$ quivers.

Representations of type $\mathbb{A}_\ell$ quivers are called {\em zigzag persistence modules} \cite{zigzag_pers}, and these specialise in the equioriented case to {\em ordinary persistence modules} \cite{oudot2015persistence}. It follows from Gabriel's theorem \cite{gabriel} that the indecomposable summands which appear in the decomposition \eqref{eq:krullschmidt} of a nonzero $V \in \Rep(Q)$ have a particularly simple form when $Q$ is of type $\mathbb{A}_\ell$. Each such indecomposable corresponds to a subinterval $[a,b] \subset [0,\ell]$ with integral endpoints. Recalling that $\mathbb{F}$ is the ground field over which all of our vector spaces are defined, the indecomposable $\Int_\tau[a,b] \in \Rep(Q)$ associated to $[a,b]$ has the form
	\begin{align}\label{eq:int-ab}
	\xymatrixcolsep{.35in}
	\xymatrix{
		0 \ar@{<->}[r] & \cdots \ar@{<->}[r]  & 0 \ar@{<->}[r] & \mathbb{F} \ar@{<->}[r] & \cdots \ar@{<->}[r]  & \mathbb{F}\ar@{<->}[r] & 0 \ar@{<->}[r] & \cdots \ar@{<->}[r] & 0.
	}
	\end{align}	Here the arrows point in accordance with the orientation $\tau$ of $Q$, the contiguous string of $\mathbb{F}$'s spans vertex indices $\set{a,a+1,\ldots,b-1,b}$, all maps with source and target $\mathbb{F}$ are identities, and all other vector spaces are trivial. These indecomposables are often called {\em interval} modules.
	
Explicitly, if $Q$ is a type $\mathbb{A}_\ell$ quiver with orientation $\tau$, then associated to each representation $V \in \Rep(Q)$ there exists a unique finite set $\Barc(V)$ consisting of subintervals $[a,b] \subset [0,\ell]$ with $a \leq b$ integers, and a unique function $\Barc(V) \to \mathbb{N}$ sending each $[a,b]$ to its multiplicity $d_{ab}$, so that there is an isomorphism
\begin{align} \label{equ:gabriel}
   V \simeq \bigoplus_{[a,b]} \left(\Int_\tau[a,b]\right)^{d_{ab}}.
\end{align}
Here the direct sum ranges over $[a,b] \in \Barc(V)$. Thus, a central charge $\alpha:\VS \to \bbR$ is complete for $Q$ in the sense of Definition \ref{def:complete} if and only if the multiplicity function $[a,b] \mapsto d_{ab}$ of every $V \in \Rep(Q)$ can be recovered from the HN type $\HNtype{\alpha}{V}$.

\subsection{Characterising complete central charges}

 The first step in our quest to describe all complete central charges of $Q$ is a converse to Lemma \ref{lemma:completenesscondition}. Throughout, we fix a quiver $Q = (\sfn,\tfn:\ES \to \VS)$ of type $\mathbb{A}_\ell$ and denote its orientation by $\tau$.

\begin{prop}\label{prop:completeequivalancetypeA}
A central charge $\alpha:\VS \to \bbR$ is complete for $Q$ if every indecomposable $\Int_\tau[a,b]$ in $\Rep(Q)$ is $\alpha$-stable.
\end{prop}

\begin{proof} Let $V \in \Rep(Q)$ have the decomposition \eqref{equ:gabriel}; we seek to establish that the multiplicities $d_{ab}$ which appear in this decomposition can be recovered from $\HNtype{\alpha}{V}$. To this end, let $\lambda_1 > \lambda_2 > \cdots > \lambda_k$ be the collection of all slopes contained in the set
\[
\set{\mu_\alpha(\Int_\tau[a,b]) \mid [a,b]  \in \Barc(V)}.
\]\
For each $i$ in $\set{1,\ldots,k}$ we denote by $B_i \subset \Barc(V)$ the subset consisting of all $[a,b] $ for which $\mu_\alpha(\Int_\tau[a,b]) \geq \lambda_i$. Consider the filtration $V^\bullet$ of $V$ given by
\[
V^i := \bigoplus_{[a,b]  \in B_i} \Int_\tau[a,b]^{d_{ab}}.
\]
By construction, the quotient $V^{i}/V^{i-1}$ is a direct sum of stable representations with $\alpha$-slope equal to $\lambda_i$. Now Corollary \ref{cor:directsumss} and uniqueness  (described in Theorem \ref{thm:HN_existence}) ensure that $V^\bullet$ is the HN filtration of $V$ along $\alpha$. 

For each $b \in \set{0,1,\ldots,\ell}$, define $\phi_b:\set{0,\ldots,b} \to \bbR$ as $\phi_b(a) := \mu_\alpha(\Int_\tau[a,b])$. We claim that these maps are injective: given $a' < a$, there are two cases to consider, depending on the orientation of $e_a \in \ES$ (or equivalently, on the value of $\tau_a \in \set{0,1}$). If $\tfn(e_a) = x_a$, then there exists a monomorphism $\Int_\tau[a,b] \subset \Int_\tau[a',b]$ and the stability of the latter representation guarantees $\phi_b(a) < \phi_b(a')$. On the other hand, if $\tfn(e_a) = x_{a-1}$ then $\Int_\tau[a,b]$ is a quotient of $\Int_\tau[a',b]$: we have a short exact sequence in $\Rep(Q)$ of the form
\[
\xymatrixcolsep{.4in}
\xymatrix{
0 \ar@{->}[r] & \Int_\tau[a',a-1] \ar@{->}[r] &  \Int_\tau[a',b] \ar@{->}[r] & \Int_\tau[a,b] \ar@{->}[r] & 0.
}
\]
An appeal to Lemma \ref{lemma:sumstability} along with the $\alpha$-stability of $\Int_\tau[a',b]$ gives $\phi_b(a) > \phi_b(a')$. In both cases we obtain $\phi_b(a) \neq \phi_b(a')$ for $a' < a$, whence $\phi_b$ is injective as claimed. 

Given this injectivity, for each fixed $i$ we may order the elements of $B_i$ as 
\[
B_i = \set{[a_1, b_1], [a_2, b_2], \ldots, [a_n, b_n]},
\] where $b_1 > \cdots > b_{n}$. Now the multiplicity $d_{a_1 b_1}$ is precisely the dimension of $V^i/V^{i-1}$ at any vertex $x_j \in \VS$ where $b_2 < j \leq b_1$. Proceeding inductively, we similarly recover the multiplicities $d_{a_kb_k}$ for all $1 < k \leq n$.
\end{proof}

The preceding result, combined with Lemma \ref{lemma:completenesscondition}, confirms that a central charge $\alpha$ for $Q$ is complete if and only if every indecomposable $\Int_\tau[a,b] \in \Rep(Q)$ is $\alpha$-stable. The main result of \cite{totalstability} is a complete characterisation of such central charges in terms of two functions: let $\chi, \eta \colon \set{0,1, \ldots ,\ell} \mapsto \mathbb{N}$ be defined inductively as follows. Beginning with $\chi(0)= 0$ and $\eta(0) = 0$, for each $i > 0$ we set
\[
\chi(i+1) = \begin{cases}
\chi(i)+1 & \text{if } \tau_i = 1 \\
\chi(i) & \text{if } \tau_i = 0
\end{cases}
\quad \text{and} \quad
\eta(i+1) = \begin{cases}
\eta(i)+1 & \text{if } \tau_i = 0 \\
\eta(i) & \text{if } \tau_i = 1
\end{cases}
\]
For instance, when $\tau = 1101$, the function $\chi$ takes on values $(0,1,2,2,3)$ while the function $\eta$ takes on values $(0,0,0,1,1)$ for inputs $(0,1,2,3,4)$:
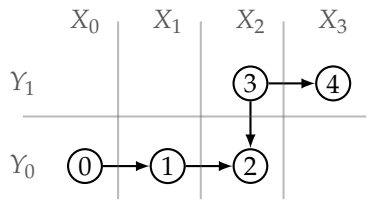
\[
{\footnotesize
\begin{tikzpicture}[
roundnode/.style={circle,draw, inner sep=1.5pt},scale=1.1,thick,->,>=latex]

\node[roundnode] (0) at (0,0){$0$};
\node[roundnode] (1) at (1,0){$1$};
\node[roundnode] (2) at (2,0){$2$};
\node[roundnode] (3) at (2,1){$3$};
\node[roundnode] (4) at (3,1){$4$};

\draw (0)--(1);
\draw (1)--(2);
\draw (3)--(2);
\draw (3)--(4);

\foreach \x in {0,1,2}{
\draw[gray!25!black!75,opacity=0.4,-]  (\x+0.4,1.75)--(\x+0.4,-0.4);
\node[gray!25!black!75] (lab\x) at (\x,1.75) {$X_{\x}$};
}
\node[gray!25!black!75] (lab3) at (3,1.75) {$X_{3}$};
\draw[gray!25!black!75,opacity=0.4,-]  (-0.75,0.6)--(3.4,0.6);
\node[gray!25!black!75] (lab3) at (-0.75,0) {$Y_{0}$};
\node[gray!25!black!75] (lab3) at (-.75,1) {$Y_{1}$};
\end{tikzpicture}}\]
For each $k \in \mathbb{N}$ we have the level sets $X_k := \set{i \mid \chi(i) = k}$ and $Y_k := \set{i \mid \eta(i) = k}$, both of which are subintervals of $\set{0,\ldots,\ell}$ as $\chi $ and $\eta$ are monotone. Writing $X_k = [a_k,b_k]$ and $Y_k = [a'_k,b'_k]$ for each $k$, we are able to state \cite[Theorem 1.13]{totalstability}.

\begin{theorem}\label{theorem:totalstability}
All indecomposables $\Int_\tau[a,b]$ in $\Rep(Q)$ are $\alpha$-stable for a given central charge $\alpha:\VS \to \bbR$ if and only if the following inequalities hold:
\begin{alignat*}{7} 
    &\mu_\alpha(\Int_\tau[a_0,b_0]) &~>~& \mu_\alpha(\Int_\tau[a_1,b_1]) &~>~& \cdots &~>~& \mu_\alpha(\Int_\tau[a_{\chi(\ell)},b_{\chi(\ell)}]), \\
     &\mu_\alpha(\Int_\tau[a'_0,b'_0]) &~<~& \mu_\alpha(\Int_\tau[a'_1,b'_1]) &~<~& \cdots &~<~& \mu_\alpha(\Int_\tau[a'_{\eta(\ell)},b'_{\eta(\ell)}]). 
\end{alignat*}
\end{theorem}

It is also shown in \cite{totalstability} that, for every possible orientation $\tau$, ({\bf a}) this is a {\em minimal} set of inequalities for characterising those central charges along which all indecomposables are stable, and ({\bf b}) the set of all such central charges defines an non-empty open subset in $\mathbb R ^{Q_0}$ which is linearly equivalent to $\mathbb R\times \mathbb R_{>0}^{Q_1}$. These results, when combined with our Proposition \ref{prop:completeequivalancetypeA} and Lemma \ref{lemma:completenesscondition}, completely describe all complete central charges for $\mathbb{A}_\ell$ quivers.

\begin{theorem} \label{theorem: mainAl}
Given an integer $\ell \geq 0$, let $Q$ be a quiver of type $\mathbb{A}_\ell$ and orientation $\tau$. The set of complete central charges for $Q$ is nonempty, and consists precisely of those $\alpha:\VS \to \bbR$ which satisfy the inequalities from Theorem \ref{theorem:totalstability}.
\end{theorem}

\noindent We note here that the set of complete central charges admits a particularly appealing description in the case where $Q$ is equioriented. 

\begin{cor}\label{cor:comp_ord_pmod}
 For ordinary persistence modules, a central charge $\alpha$ is complete if and only if the inequality $\alpha(x_i) > \alpha(x_{i+1})$ holds for all $i \in \set{0,1,\ldots,\ell-1}$.
 \end{cor}

 \begin{proof}
   If $\tau = 11 \ldots 1$, then the function $\chi:\set{0,1,\ldots,\ell} \to \mathbb{N}$ is given by $\chi(i) = i$, whereas the function $\eta$ is identically zero. We therefore seek any $\alpha:\VS \to \bbR$ which satisfies $\mu_\alpha(\Int_\tau[i,i]) > \mu_\alpha(\Int_\tau[i+1,i+1]) $ for all $i$. By \eqref{eq:alphaslope} and \eqref{eq:int-ab}, this string of inequalities reduces to
\begin{align}\label{eq:totallystablecondsstandardpers}
\alpha(x_0) > \alpha(x_1) > \cdots >\alpha(x_{\ell-1}) > \alpha(x_\ell),
\end{align}
as desired.
\end{proof}

\section{HN types of multiparameter persistence modules} \label{sec:multi}

Finding good invariants for multiparameter persistence modules is a central challenge in topological data analysis. In this section we prove a  generalisation  of  Corollary \ref{cor:comp_ord_pmod} to the multiparameter setting. 

\subsection{Multiparameter persistence modules as equalised representations}
Let $Q =(\sfn,\tfn:\ES\to\VS)$ be an acyclic quiver. Its source and target maps may be extended from edges to paths $\gamma = (e_1,\ldots,e_k)$ by setting $\sfn(\gamma) := \sfn(e_1)$ and $\tfn(\gamma) := \tfn(e_k)$. Given a representation $V \in \Rep(Q)$, there is a distinguished linear map $V_\gamma:V_{\sfn(\gamma)} \to V_{\tfn(\gamma)}$ induced by the composite
\[
V_\gamma :=  V_{e_k} \circ V_{e_{k-1}} \circ \cdots \circ V_{e_2} \circ V_{e_1}.
\]

\begin{definition}
    We say that $V \in \Rep(Q)$ is {\bf equalised} if the following property holds: for any pair of vertices $x,y \in \VS$ and any pair of paths $\gamma,\gamma'$ with common source $x$ and common target $y$, the composite maps $V_\gamma$ and $V_{\gamma'}$ are identical. Let $\Rep_\eq(Q) \subset \Rep(Q)$ be the full subcategory spanned by equalised representations.
\end{definition} 

\begin{ex}
 A large class of interesting equalised representations arises in the study of {\em cellular sheaves} \cite{curry2014sheaves}. Every such sheaf $\mathcal{F}$ on a regular CW complex $X$ is a functor from the face-ordered poset of cells $(X,\prec)$ to the category $\textbf{Vec}(\K)$ of $\K$-vector spaces. The {\em Hasse graph} of $X$ is the quiver $Q(X)$ whose vertices correspond bijectively to the cells of $X$, with a unique edge $\sigma \to \tau$ being present whenever $\sigma$ is a face of $\tau$ of codimension one. Any given sheaf $\mathcal{F}:(X,\prec) \to \textbf{Vec}(\K)$ on $X$ induces a representation $V(\mathcal{F})$ of $Q(X)$ as follows: every vertex $\sigma$ is assigned the vector space $\mathcal{F}(\sigma)$ and every edge $\sigma \to \tau$ is assigned the linear map $\mathcal{F}(\sigma \prec \tau)$. The fact that $\mathcal{F}$ is a functor directly implies that $V(\mathcal{F})$ is equalised.
 \end{ex}

\begin{definition} \label{def:floworder}
    The {\bf flow partial order} on vertices $\VS$ of the acyclic quiver $Q$ is defined as follows: given $x$ and $y$ in $\VS$, we have $x \leq y$ if either $x = y$ or if there exists a path $\gamma$ in $Q$ with $\sfn(\gamma) = x$ and $\tfn(\gamma) = y$. 
\end{definition}

Given $V \in \Rep_\eq(Q)$ and a pair of vertices $x \leq y$ in $\VS$, we write $V_{x \leq y}:V_x \to V_y$ to indicate the map defined by any path $\gamma$ from $x$ to $y$, with the understanding that this map is the identity for $y = x$. Since $V$ is equalised, this is well-defined and it follows that the image of $V_{x \leq y}$ is isomorphic to $\subrep{V_x}_y$ (from Definition \ref{def:spanrep}). Thus, the value of $\rho_V(x,y)$ is precisely the rank of $V_{x \leq y}$ when $V$ is equalised. This is the genesis of the terminology of the rank invariant, which was introduced in \cite{multi} to study certain equalised representations of grid quivers, described below.

Let us fix a vector $L = (\ell_1,\ell_2, \ldots, \ell_d)$ of $d \geq 2$ integers, with each $\ell_i \geq 1$. Here we consider the case where $Q = (\sfn,\tfn:\ES \to \VS)$ is the $d$-dimensional {\em grid quiver} of shape $L$, defined as follows. Its vertices $x_p$ are indexed by all $p$ lying in the product 
\[
\Lambda(L) := \prod_{i=1}^d\set{0,1,\ldots,\ell_i},
\] and there exists a unique edge $x_p \to x_q$ whenever $q-p$ is a standard basis vector of $\bbR^d$. Here, for instance, is the quiver of shape $L = (\ell,2)$ for arbitrary $\ell \geq 1$:
\[
\xymatrixcolsep{.6in}
\xymatrixrowsep{.4in}
\xymatrix{
x_{0,2} \ar@{->}[r] & x_{1,2} \ar@{->}[r]_{} & \cdots \ar@{->}[r]_{} & x_{\ell-1,2} \ar@{->}[r]_{} & x_{\ell,2} \\
x_{0,1} \ar@{->}[r]^{} \ar@{->}[u]_{} & x_{1,1} \ar@{->}[r]^{} \ar@{->}[u]_{} & \cdots \ar@{->}[r]^{}  & x_{\ell-1,1} \ar@{->}[r]^{} \ar@{->}[u]_{}& x_{\ell,1} \ar@{->}[u]_{} \\
x_{0,0} \ar@{->}[r]^{} \ar@{->}[u]_{} & x_{1,0} \ar@{->}[r]^{} \ar@{->}[u]_{} & \cdots \ar@{->}[r]^{}  & x_{\ell-1,0} \ar@{->}[r]^{} \ar@{->}[u]_{}& x_{\ell,0} \ar@{->}[u]_{} 
}
\]
Equalised representations of $d$-dimensional grid quivers are also referred to as  {\bf $d$-parameter persistence modules} \cite{multi}. We note that every grid quiver $Q$ contains an embedded copy of the grid quiver of shape $L = (1,1)$, and that both the representations $W$ and $W'$ which appeared in the proof of Theorem \ref{thm:rankfromHN} are equalised. Thus, we obtain the following consequence.

\begin{cor}\label{cor:skyrank}
    Given any integer $d \geq 2$, let $Q$ be the grid quiver corresponding to some integer vector $L = (\ell_1,\ldots,\ell_d)$ with each $\ell_i \geq 1$. 
    \begin{enumerate}
        \item The skyscraper invariant $\delta_V$ of  $V \in \Rep_\eq(Q)$ determines its rank invariant (via the formula in Theorem \ref{thm:rankfromHN}).
        \item There exist representations $W$ and $W'$ in $\Rep_\eq(Q)$ which have identical rank invariant and satisfy $\delta_W \neq \delta_{W'}$.
    \end{enumerate}
\end{cor}

\subsection{Rectangle-decomposable representations}
In general, grid quivers are of wild representation type and one cannot expect to obtain a tractable classification of all indecomposable objects in $\Rep_\eq(Q)$. One can, however, impose a higher-dimensional analogue of \eqref{equ:gabriel} by passing to the subset of {\em rectangle-decomposable} representations, which we describe below. 

As before, let $Q = (\sfn,\tfn:\ES \to \VS)$ be the grid quiver of shape $L = (\ell_1,\ell_2, \dots,\ell_d)$ for $d\geqslant 2$ and all $\ell_i \geq 1$. Given any subset $P \subset \Lambda(L)$, we write $\Int[P]$ to denote the representation of $Q$ which assigns vector spaces
\[
\Int[P]_{x_p} = \begin{cases}
\K & \text{if } p \in P, \\
0 & \text{otherwise};
            \end{cases} 
\] 
the linear map associated to each edge is the identity whenever both source and target spaces are $\K$, and it must necessarily equal zero otherwise. By a {\bf rectangle representation} we mean $\Int[R]$, where $R := [a_1,b_1] \times\dots \times [a_d,b_d]$ for some $[a_i,b_i] \subset [0,\ell_i]$ is a rectangle inside $\Lambda(L)$. We note that $\Int[R]$ is always equalised when $R$ is a rectangle. We write $\Rect(Q)$ for the full subcategory of $\Rep_\eq(Q)$ spanned by objects which are (isomorphic to) direct sums of rectangle representations. The rank invariant is complete when restricted to this subcategory   \cite{botnan2022rectangle, clause2022discriminating};
and by Corollary \ref{cor:skyrank}, so is the skyscraper invariant. 

Our goal in this subsection is to prove a much sharper result --- we extend Corollary \ref{cor:comp_ord_pmod} to the category of rectangle-decomposable representations of arbitrary dimension $d$ by classifying the set of complete central charges. For this purpose, it is necessary to exclude from consideration a finite union of hyperplanes in the vector space of central charges:
\begin{align}\label{eqn:hyperplanes}
\mathcal H:=\bigcup_{R\neq R'} \set{\alpha:Q_0 \to \bbR \mid \mu_\alpha(\Int[R])=\mu_\alpha(\Int[R'])},
\end{align}
where $R$ and $R'$ range over distinct rectangles in $\Lambda(L)$. The following result serves to justify this exclusion. 

\begin{prop}\label{lmm:stab_to_complet_rect }
    If $\alpha\notin \mathcal H$ is a central charge for which each rectangle representation $\Int[R] \in \Rep_\eq(Q)$ is $\alpha$-stable, then $\HNtype{\alpha}\bullet$ is complete on $\Rect(Q)$.
\end{prop}
\begin{proof}
    Given $V\in\Rect(Q)$, consider its decomposition 
    \[
    V\simeq\bigoplus_R \Int[R]^{m_R},
    \] 
    where the direct sum is indexed over all subrectangles $R \subset \Lambda(L)$, of which only finitely many have multiplicity $m_R > 0$. It suffices to recover these multiplicities from the HN type of $V$ along $\alpha$. By Proposition \ref{prop:slopesum}, for each real number $c \in \bbR$ we have
    \[    
    \HNtype{\alpha}{V}(c) = \sum_{\mu_\alpha(R)=c}m_{R}\cdot\udim_{\Int[R]}
    \]
    Since $\alpha \notin \mathcal{H}$ by assumption, the multiplicity $m_{R'}$ of any rectangle $R'$ can be obtained by letting $c = \mu_\alpha(R')$, so the sum simplifies to $\HNtype{\alpha}{V}(c) = m_{R'} \cdot \udim_{\Int[R']}$.
\end{proof}

The following is a multiparameter generalisation of Corollary \ref{cor:comp_ord_pmod}.

\begin{theorem}\label{thm:complete_param_rect_equiv1}
Let $Q$ be a grid quiver of shape $L = (\ell_1,\ell_2,\ldots,\ell_d)$ for any $\ell_i \geq 1$. A central charge $\alpha \notin \mathcal H$ is complete on $\Rep_\eq(Q)$ if and only if it satisfies the inequality $\alpha \circ \sfn(e) > \alpha \circ \tfn(e)$ for each edge $e \in \ES$.    
\end{theorem}

The remainder of the section will be occupied by the proof.

\vskip .1in
We first establish a technical result in lattice theory.
Let us recall the flow partial order on $\VS$ from Definition \ref{def:floworder}. A subset of vertices $U\subset Q_0$ is said to be {\bf up-closed} if $x \in U$ and $y \geq x$ implies $y \in U$. Given an arbitrary nonempty subset $A \subset \VS$, we denote by $A^+$ the smallest up-closed subset of $\VS$ containing $A$. The poset $(\VS,\leqslant)$ has a structure of finite lattice with $\wedge$ and $\vee$ given by applying respectively $\min$ and $\max$ coordinate-wise. One can check that this lattice is distributive (the distributive law is  true in each coordinate), and hence that it satisfies the following standard inequality \cite[Corollary 6.1.3]{alon2008probabilistic}.

\begin{prop}\label{prop:4funtheorem}
    Let $(L,\wedge,\vee)$ be a finite distributive lattice. For any subsets $X,Y \subset L$, we have the inequality
    \[ 
    |X|\cdot |Y| \leq |X\wedge Y|\cdot |X\vee Y|, \text{ where:}
\]
\begin{enumerate}
    \item $|\bullet|$ denotes cardinality, 
    \item $X\vee Y:= \set{x\vee y\mid x\in X \text{ and } y\in Y}$, and
    \item $X \wedge Y:= \set{x\wedge y\mid x\in X \text{ and }y\in Y}$.
\end{enumerate}
\end{prop}

\noindent We will use this combinatorial inequality to establish the following result about up-closed subsets of $Q_0$. 

\begin{lemma}\label{lmm:techn_flow}
     Let $U \subset \VS$ be an up-closed subset with complement $D := Q_0 \setminus U$. Then, for all  subsets $A \subset D$ we have       
     \[
       |A| \cdot |U| \leq |D| \cdot |U \cap A^+|
    \]
       where $|\bullet|$ denotes cardinality.
\end{lemma}
\begin{proof}
Since, $(A^+\cap D)^+=A^+$, it suffices to establish the desired inequality with $A$ replaced by the larger set $A^+\cap D$. Since $A^+\cap D$ equals $ A^+\setminus(U\cap A^+)$, this inequality becomes
\[ 
\left(|A^+|-|U\cap A^+|\right)\cdot |U|\leqslant \left(|\VS|-|U|\right)\cdot |U\cap A^+|,
\]
which is equivalent to
\[ |A^+|\cdot |U|\leqslant |\VS|\cdot |U\cap A^+|.
\]
Since $U$ and $A^+$ are up-closed, we have $U\cap A^+=U\vee A^+$. Applying Proposition \ref{prop:4funtheorem} with subsets $X = U$ and $Y = A^+$ of $L = \VS$, we obtain
\[ |A^+|\cdot |U|\leq |U\wedge A^+|\cdot |U\vee A^+|\leq |\VS|\cdot |U\cap A^+|,\]
as desired.
\end{proof}

\vskip .1in
 The main tool in the proof of our generalisation of Corollary \ref{cor:comp_ord_pmod} to $\Rect(Q)$ is the following {\bf max-flow/min-cut theorem} \cite[Chapter III.1]{bollobas}. 

\begin{theorem}\label{thm:mfmc}
    Let $\Phi = (\sigma,\tau:\Phi_1 \to \Phi_0)$ be a quiver whose vertex set contains a distinguished source $s_* \not\in \tau(\Phi_1)$ and target $t_* \not\in \sigma(\Phi_1)$, and let $\kappa:\Phi_1 \to [0,\infty]$ be a function defined on edges. The maximum value attained by a $\kappa$-flow equals the minimum $\kappa$-capacity of a cut separating $s_*$ from $t_*$. 
\end{theorem}
\noindent Recall that a $\kappa$-flow on $\Phi$ is a map $f:\Phi_1 \to [0,\infty]$ satisfying two constraints:
\begin{enumerate}
 \item $f(\epsilon) \leq \kappa(\epsilon)$ for all $\epsilon \in \Phi_1$, and
 \item $\sum_{\sigma(\epsilon)=x} f(\epsilon) = \sum_{\tau(\epsilon)=x} f(\epsilon)$ for all $x \in \Phi_0 \setminus \set{s_*,t_*}$.
 \end{enumerate}
 The {\em value} of $f$ is the sum $\nu(f) := \sum_{\sigma(\epsilon) = s_*}f(\epsilon)$; by the second constraint above, $\nu(f)$ also equals $\sum_{\tau(\epsilon)=t_*}f(\epsilon)$. On the other hand, an $(s_*,t_*)$-cut is any subset $E \subset \Phi_1$ whose removal disconnects $s_*$ from $t_*$; the {\em $\kappa$-capacity} of such a cut is  $c(E) := \sum_{\epsilon \in E}\kappa(\epsilon)$. We are now able to characterise complete central charges for $Q$ which do not lie in the union of hyperplanes $\mathcal H$ from \eqref{eqn:hyperplanes}.

 \begin{theorem}\label{thm:complete_param_rect_equiv}
Let $Q$ be a grid quiver of shape $L = (\ell_1,\ell_2,\ldots,\ell_d)$ for any $\ell_i \geq 1$. A central charge $\alpha \notin \mathcal H$ is complete on $\Rep_\eq(Q)$ if and only if it satisfies the inequality $\alpha \circ \sfn(e) > \alpha \circ \tfn(e)$ for each edge $e \in \ES$.    
\end{theorem}
 
\begin{proof}
   If the inequality $\alpha \circ \sfn(e) \leq \alpha \circ \tfn(e)$ holds for some edge $e$, then we obtain $\mu_\alpha(\Int[\set{t(e)}]) \leq \mu_{\alpha}(\Int[\set{s(e),t(e)}])$. We now have from Lemma \ref{lemma:completenesscondition} that the restriction of $\alpha$ to the $\mathbb{A}_2$ quiver $s(e) \to t(e)$ is not complete. It is straightforward  to confirm that as a consequence $\alpha$ is not complete on $\Rect(Q)$. The remainder of the argument will be devoted to the converse implication --- assuming that $\alpha \circ s(e) > \alpha \circ t(e)$ holds for all $e \in \ES$, we will show that $\alpha$ is complete.

    By Proposition \ref{lmm:stab_to_complet_rect }, it is enough to prove that each rectangle representation is $\alpha$-stable. By passing to the subquiver induced by vertices lying within any such rectangle, it suffices to assume that the rectangle representation at hand is $V=\Int[\Lambda(L)]$. Consider a subrepresentation $V' \subsetneq \Int[\Lambda(L)]$; this assigns vector spaces of dimension at most 1 with all nontrivial maps being isomorphisms. Thus, $V'$ has the form $\Int[U]$ for some up-closed proper subset $U \subsetneq \VS$. Let $\set{u_1,u_2,\ldots,u_m}$ be the vertices lying in $U$, and similarly let $\set{d_1,d_2,\ldots,d_n}$ be the vertices lying in its complement $D:= \VS\setminus U$. 
     
      Construct a new quiver $\Phi = (\sigma,\tau:\Phi_1 \to \Phi_0)$ as follows: its vertex set $\Phi_0$ consists of $Q_0$ along with two additional vertices $s_*$ and $t_*$. The edge set $\Phi_1$ is built in two stages as follows: first we insert a unique edge from $s_*$ to each vertex of $D$, and similarly a unique edge from each vertex of $U$ to $t_*$, as depicted below:
 
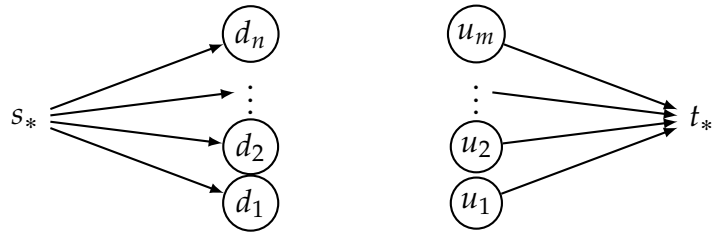
\begin{equation*}
\begin{array}{c}\begin{tikzpicture}[
roundnode/.style={circle,draw, inner sep=1.8pt},scale=0.75,thick,->,>=latex]
    \node (s) at (-2,2.5){$s_*$};

    \node[roundnode] (s1) at (2,1){$d_1$};
     \node[roundnode] (s2) at (2,2){$d_2$};
    \node (s3) at (2,3){$\vdots$};
    \node[roundnode] (s4) at (2,4){$d_{n}$};
   
     \draw (s) -- (s1);
    \draw (s)-- (s2);
    \draw (s)--(s3);
    \draw (s) -- (s4);

    \node (t) at (10,2.5){$t_*$};
        \node[roundnode] (t1) at (6,1){$u_1$};
     \node[roundnode] (t2) at (6,2){$u_2$};
    \node (t3) at (6,3){$\vdots$};
    \node[roundnode] (t4) at (6,4){$u_m$};

     \draw (t1) -- (t);
    \draw (t2)--(t);
    \draw (t3)--(t);
    \draw (t4) -- (t);
    
\end{tikzpicture}\end{array}\end{equation*}
In the second step, we add a unique edge $d_i \to u_j$ whenever $d_i \leq u_j$ holds in $\VS$. Define $\kappa:\Phi_1 \to [0,\infty]$ by:
\[
\kappa(\epsilon) = \begin{cases}
                      1/n & \text{if } \sigma(\epsilon) = s_* \\
                      1/m & \text{if } \tau(\epsilon) = t_* \\
                      +\infty & \text{otherwise}.
                    \end{cases}
\]

{\bf Claim: } Every $(s_*,t_*)$-cut $E \subset \Phi_1$ has $\kappa$-capacity $c(E) \geq 1$. \\
\noindent Given such a cut, let $S$ and $T$ denote the vertices lying in the component of $s_*$ and $t_*$ respectively after the edges of $E$ have been removed. We may safely assume that $E$ contains no edges of the form $d_i \to u_j$ since that would immediately force $c(E) = \infty$. Therefore, writing $A := S \cap D$ and $B := T \cap U$, we know that $A^+ \cap B$ is empty because the removal of $E$ must separate $s_*$ from $t_*$. As a result, we have
\[
c(E) = \frac{|D \setminus A|}{|D|} + \frac{|U \setminus B|}{|U|}.
\]
Using the fact that $U \setminus B$ contains $A^+ \cap U$ followed by Lemma \ref{lmm:techn_flow}, we have 
\[
\frac{|U \setminus B|}{|U|} \geq \frac{|U \cap A^+|}{|U|} \geq \frac{|A|}{|D|} = 1-\frac{|D\setminus A|}{|D|}.
\]
Using this bound in our expression for $c(E)$ given above establishes the claim.

Returning to the main argument, we have by Theorem \ref{thm:mfmc} that $\Phi$ admits a $\kappa$-flow $f:\Phi_1 \to [0,\infty]$ of value $\nu(f) \geq 1$. Select any such $f$ and note that it must evaluate to $1/n$ on each edge $s_* \to d_i$ and to $1/m$ on each edge $u_j \to t_*$, whence its value $\nu(f)$ is exactly $1$. Define $F:D \times U \to \bbR_{\geq 0}$ by 
\[
F(d_i,u_j) = \begin{cases}
                f(d_i \to u_j) & \text{if } d_i \leq u_j \text{ holds in } \VS \\
                0 & \text{ otherwise.}
         \end{cases}
\]
Since $\alpha \circ s(e) > \alpha \circ t(e)$ holds for each edge $e \in \ES$, and since $F$ takes strictly positive values on at least some $(d_i,u_j)$ pairs, we have 
\[
\sum_{i=1}^n \sum_{j=1}^m F(d_i,u_j) \cdot \alpha(d_i) > \sum_{i=1}^n \sum_{j=1}^m F(d_i,u_j) \cdot \alpha(u_j).
\] 
Using the fact that $f$ is a $\kappa$-flow, for each $i$ we have $\sum_{j=1}^mF(d_i,u_j) = 1/n$, and similarly for each $j$ we have $\sum_{i=1}^nF(d_i,u_j) = 1/m$. Thus, the inequality above is $\phi_\alpha(\Int[D]) > \phi_\alpha(\Int[U])$. Finally, the desired inequality $\phi_\alpha(V') < \phi_\alpha(V)$ follows from Lemma \ref{lemma:sumstability} applied to the short exact sequence $0\to V'\to V\to \Int[D]\to 0$.
\end{proof}

\section{HN types of nestfree ladder persistence modules} \label{sec:ladders}

Fix an integer $\ell \geq 1$. Here we will be concerned with certain equalised representations of the following {\em ladder} quiver $Q = (\sfn,\tfn:\ES \to \VS)$ of length $\ell$:
\[
\xymatrixcolsep{.8in}
\xymatrixrowsep{.6in}
\xymatrix{
x_0^+ \ar@{->}[r]^{e_1^+} \ar@{->}[d]_{e_0} & x_1^+ \ar@{->}[r]^{e_2^+} \ar@{->}[d]_{e_1} & \cdots \ar@{->}[r]^{e_{\ell-1}^+}  & x_{\ell-1}^+ \ar@{->}[r]^{e_{\ell}^+} \ar@{->}[d]_{e_{\ell-1}}& x_\ell^+ \ar@{->}[d]_{e_\ell} \\
x_0^- \ar@{->}[r]_{e_1^-} & x_1^- \ar@{->}[r]_{e_2^-} & \cdots \ar@{->}[r]_{e^-_{\ell-1}} & x_{\ell-1}^- \ar@{->}[r]_{e^-_\ell} & x_\ell^-
}
\]
Equalised representations of such quivers are sometimes called {\em ladder persistence modules}; these are precisely 2-parameter persistence modules of shape $L = (\ell,1)$, but it is customary to represent them with vertical arrows pointing down instead of up. In particular, they arise when studying the morphisms of ordinary persistence modules. The authors of \cite{escolar2016persistence} established that $\Rep_\eq(Q)$ is of finite type for $\ell \leq 3$ and completely classified its indecomposable objects via Auslander-Reiten theory. In contrast, the last three authors study in \cite[Section 5]{barcodebases} a full subcategory of $\Rep_\eq(Q)$ which turns out to be representation finite regardless of $\ell$. These are the nest-free representations.

\subsection{Nestfree representations}

 Given $V \in \Rep_\eq(Q)$, we let $V^+$ and $V^-$ denote its restrictions to the top and bottom rows of $Q$ respectively. Since both $V^\pm$ are representations of the (equioriented) $\mathbb{A}_\ell$ quiver, they admit direct sum decompositions into interval modules as described in \eqref{equ:gabriel}. We say that such an $\mathbb{A}_\ell$ representation $W$ admits a pair of {\em strictly nested intervals} whenever there exist $[a,b]$ and $[c,d]$ in $\Barc(W)$ satisfying $a < c \leq d < b$ --- in other words, the interval $[c,d]$ lies within the interior of the interval $[a,b]$.

\begin{definition}
A representation $V \in \Rep_\eq(Q)$ is said to be {\bf nestfree} if neither $V^+$ nor $V^-$ admits a pair of strictly nested intervals. 
\end{definition}

We will examine the full subcategory of $\Rep_\eq(Q)$ spanned by nestfree representations, which is denoted $\Rep_{\nf}(Q)$ henceforth. Its indecomposable objects were completely classified in \cite[Theorem 5.3]{barcodebases}, which we paraphrase below.

\begin{theorem} \label{thm:ladderindecs}
Up to isomorphism, the indecomposable objects of $\Rep_\nf(Q)$ have one of three possible forms:
\begin{enumerate}
    \item Given subintervals $[a,b]$ and $[c,d]$ of $[0,\ell]$ whose endpoints satisfy $c \leq a \leq d \leq b$,  let $\bR^{a,b}_{c,d}$ be the representation $V$ for which $V^+$ is the interval module $\Int[a,b]$ while $V^-$ is the interval module $\Int[c,d]$, and all vertical maps are $1'$s whenever possible and $0$ otherwise:
    \[
    \xymatrixcolsep{.35in}
    \xymatrixrowsep{.25in}
    \xymatrix{
    & & \K  \ar@{->}[r] \ar@{->}[d]  & \K  \ar@{->}[r] \ar@{->}[d]  & \cdots \ar@{->}[r] & \K   \ar@{->}[r] \ar@{->}[d] & \cdots \ar@{->}[r]& \K  \\
    \K   \ar@{->}[r] & \cdots \ar@{->}[r]& \K  \ar@{->}[r]  & \K  \ar@{->}[r]  & \cdots \ar@{->}[r]& \K  
    }
\]    
  \item Given an interval $[a,b] \subset [0,\ell]$, let $\bR^{a,b}$ be the ladder representation $V$ for which $V^+$ equals $\Int[a,b]$ and $V^-$ is trivial, with all vertical maps necessarily being $0$:
  \[
\xymatrixcolsep{.35in}
\xymatrixrowsep{.25in}
\xymatrix{
 & & \K  \ar@{->}[r] \ar@{->}[d]  & \K  \ar@{->}[r] \ar@{->}[d] & \cdots \ar@{->}[r]  & \K  \ar@{->}[d]\\
0  \ar@{->}[r] & \cdots \ar@{->}[r]& 0 \ar@{->}[r]  & 0 \ar@{->}[r]  & \cdots \ar@{->}[r]& 0 \ar@{->}[r] & \cdots \ar@{->}[r]& 0    
}
\]
\item Finally, given an interval $[c,d] \subset [0,\ell]$, let $\bR_{c,d}$ be representation $V$ for which $V^+$ is trivial while $V^-$ is $\Int[c,d]$, so once again all vertical maps are $0$:
\[
\xymatrixcolsep{.35in}
\xymatrixrowsep{.25in}
\xymatrix{
0  \ar@{->}[r] & \cdots \ar@{->}[r]& 0 \ar@{->}[r] \ar@{->}[d] & 0 \ar@{->}[r] \ar@{->}[d]  & \cdots \ar@{->}[r] & 0 \ar@{->}[r] \ar@{->}[d]& \cdots \ar@{->}[r]& 0 \\
& & \K  \ar@{->}[r]   & \K  \ar@{->}[r]  & \cdots \ar@{->}[r]  &  \K 
}
\]
\end{enumerate}
\end{theorem}

It will be convenient in the sequel to unify notation by adopting the convention 
\begin{equation}\label{eqn:convs_ladder}
    \bR^{a,b}_{\infty,\infty}:= \bR^{a,b}\qquad \text{ and }\qquad\bR^{\infty,\infty}_{c,d}:=\bR_{c,d},
\end{equation}
so that for every $V \in \Rep_\nf(Q)$ there exist multiplicities $r_{c,d}^{a,b}\in \mathbb{N}$ satisfying:
\begin{align} \label{equ:ladderdecomp}
 V  \simeq  \bigoplus_{a,b,c,d}  \left(\bR_{c,d}^{a,b}\right)^{r^{a,b}_{c,d}},
\end{align} 
with $(a,b,c,d)$ ranging over admissible subsets of $\set{0,1,\ldots,\ell,\infty}^4$ as per \eqref{eqn:convs_ladder}. Our goal throughout the remainder of this section is to quantify the extent to which these multiplicities can be recovered from HN types. The first result in this direction is negative: for $\ell \geq 4$, there is no complete central charge $\alpha:Q_0 \to \bbR$.

\begin{prop} \label{prop:laddercounterexample}
If $Q$ has length $\ell \geq 4$, then for every central charge $\alpha:Q_0 \to \bbR$ there exists at least one indecomposable in $\Rep_\nf(Q)$ which is not $\alpha$-stable. 
\end{prop}
\begin{proof} 
It suffices to consider $\ell = 4$ which embeds into all larger ladder quivers:
\[
\xymatrixcolsep{.5in}
\xymatrixrowsep{.4in}
\xymatrix{ 
 x_0^+ \ar@{->}[r]^{} \ar@{->}[d]_{} & x_1^+ \ar@{->}[r]^{} \ar@{->}[d]_{} & x_2^+ \ar@{->}[r]^{} \ar@{->}[d]_{} & x_3^+ \ar@{->}[r]^{} \ar@{->}[d]_{}& x_4^+ \ar@{->}[d]_{} \\
x_0^- \ar@{->}[r] & x_1^- \ar@{->}[r]_{} & x_2^- \ar@{->}[r]_{} & x_3^- \ar@{->}[r]_{} & x_4^- 
}
\] Let $\alpha^\pm_j$ be the valued assigned by a given central charge $\alpha$ to the vertex $x^\pm_j$ for $0 \leq j \leq 4$. Assume, for the purposes of contradiction, that all the indecomposables in $\Rep_\nf(Q)$ are $\alpha$-stable. Consider the inequalities of slopes arising from the following inclusions of indecomposables: 
\begin{align*}
3(\alpha_1^- +\alpha_1^+) &<  2(\alpha_0^-+\alpha_1^-+\alpha_1^+) & \bR^{1,1}_{1,1} \subset \bR^{1,1}_{0,1} \\
3(\alpha_0^- +\alpha_1^-) &< (\alpha_0^- +\alpha_1^-+\alpha_1^+ +\alpha_2^++\alpha_3^+ + \alpha_4^+) & \bR_{0,1} \subset \bR^{1,4}_{0,1} \\
 4(\alpha_3^- +\alpha_3^+ + \alpha_4^+) &<  3(\alpha_2^- +\alpha_3^- +\alpha_3^+ + \alpha_4^+) & \bR^{3,4}_{3,3} \subset \bR^{3,4}_{2,3} \\
 4(\alpha_2^- +\alpha_2^+ + \alpha_3^+) &< 3(\alpha_1^- +\alpha_2^- +\alpha_2^+ + \alpha_3^+) & \bR^{2,3}_{2,2} \subset \bR^{2,3}_{1,2} \\
 4\alpha_4^+ &< (\alpha_2^- +\alpha_3^-+\alpha_3^+ +\alpha_4^+) & \bR^{4,4} \subset \bR^{3,4}_{2,3}
\end{align*}
Labelling these five inequalities as $(a), (b), \ldots, (e)$ respectively, we obtain
\[
4 (a)+4(b)+(c)+4(d)+ (e) \text{ holds if and only if } 0 < 0,
\]
which provides the desired contradiction.
\end{proof}

Note that Lemma  \ref{lemma:completenesscondition} also holds with $\Rep(Q)$ replaced by $\Rep_\nf(Q)$ since all the representations involved in its proof remain nestfree for indecomposable $I$. When combined with this Lemma, the calculation in Proposition \ref{prop:laddercounterexample} yields the following consequence.
\begin{cor}
    For $\ell \geq 4$, there is no central charge $\alpha:Q_0 \to \bbR$ that is complete on $\Rep_\nf(Q)$.
\end{cor}
\noindent Before remedying this defect by considering a larger collection of central charges, we describe another negative result which highlights the necessity of restricting our focus to nestfree representations.

\begin{prop} Let $\K$ be any field other than $\bbZ/2$, and consider a ladder quiver $Q$ of length $\ell \geq 4$. There exist non-isomorphic representations $V \not \simeq W$ in $\Rep_\eq(Q)$ such that $\HNtype{\alpha}V=\HNtype{\alpha}W$ for every central charge $\alpha:Q_0 \to \bbR$.
\end{prop}
\begin{proof}
    
For each scalar $\lambda$ in $\K$, consider $V(\lambda) \in \Rep_\eq(Q)$ given by
    \[
\xymatrixcolsep{.8in}
\xymatrixrowsep{.6in}
\xymatrix{
0 \ar@{->}[r] \ar@{->}[d] & \K  \ar@{->}[r]^-{\smat{1 \\ 0}} \ar@{->}[d]_-{\smat{1 \\ 1}}  & \K^2  \ar@{->}[r]^-{\smat{
	1 & 0\\ 0 & 1}} \ar@{->}[d]_-{\smat{1 & \lambda \\ 1 & 1}} & \K^2 \ar@{->}[d]_-{\smat{
	1 & \lambda}} \ar@{->}[r]^-{\smat{1 & 0}}  & \K \ar@{->}[d]\\
\K  \ar@{->}[r]_-{\smat{1\\0}} & \K^2 \ar@{->}[r]_-{\smat{1 & 0\\0 & 1}}& \K^2 \ar@{->}[r]_-{\smat{1 & 0}}  &  \K \ar@{->}[r] & 0 
}
\]
(See also \cite[Definition 4(2)]{wildladder}.) We note that regardless of the chosen $\lambda$, the upper row $V(\lambda)^+$ is $\Int[1,4] \oplus \Int[2,3]$ whereas the lower row $V(\lambda)^-$ is $\Int[0,3] \oplus \Int[1,2]$, so both admit nested intervals in their barcodes. Since we have assumed $\K \neq \bbZ/2$, there exist distinct nonzero scalars $\lambda \neq \mu$ in $\K$; we set $V := V(\lambda)$ and $W := V(\mu)$. Any isomorphism $\phi:V \to W$ must restrict to automorphisms $\phi^\pm:V^\pm \to W^\pm$ of the top and bottom rows. By \cite[Theorem 4.4]{barcodebases}, both $\phi^\pm$ are forced to be trivial in the chosen bases, i.e., represented by the identity matrix on each vertex. It is readily checked (along the middle vertical edge) that this collection of identity matrices does not constitute a morphism $V \to W$ in $\Rep(Q)$, whence $V$ and $W$ are non-isomorphic in the subcategory $\Rep_\eq(Q)$. On the other hand, there is an evident bijection between the subrepresentations of $V$ and those of $W$ that preserves dimension vectors, and hence, $\alpha$-semistability for any choice of central charge $\alpha:Q_0 \to \bbR$. This yields $\HNtype{\alpha}{V} = \HNtype{\alpha}{W}$, as claimed.
\end{proof}

\subsection{A complete set of central charges}

Fix a ladder quiver $Q$ of length $\ell \geq 1$. Our goal in this subsection is to prove the following result.

\begin{theorem} \label{thm:mainladder}
There exists a finite collection of central charges $A$ for which the HN type $\HNtype{A}{\bullet}$ is complete on $\Rep_\nf(Q)$.
\end{theorem}

We let $\inds$ denote the set of (isomorphism classes of) indecomposable objects of $\Rep_\nf(Q)$ from Theorem \ref{thm:ladderindecs}; we will also adopt the convention \eqref{eqn:convs_ladder} for describing these indecomposables as $\bR^{a,b}_{c,d}$ for certain $a,b,c,d$ in the set $[\ell]_\infty := \set{0,1,\ldots, \ell,\infty}$.

\begin{rem} \label{rem:indecinc} Here are all the possible strict inclusions $I' \subsetneq I$ in $\inds$:
\begin{center}
\renewcommand{\arraystretch}{1.4}
\begin{tabular}{clcl} 
(1) & $\bR^{a',b}_{\infty,\infty} \subsetneq \bR^{a,b}_{\infty,\infty}$ & \qquad & if  $a < a'$ \\
(2) & $\bR^{\infty,\infty}_{c',d}  \subsetneq  \bR^{\infty,\infty}_{c,d}$ & \qquad & if $c< c'$ \\
(3) & $\bR^{a',b}_{\infty,\infty} \subsetneq \bR^{a,b}_{c,d}$ & \qquad & if $d < a'$ \\
(4) & 
$\bR^{\infty,\infty}_{c',d}  \subsetneq \bR^{a,b}_{c,d}$ &  \qquad &  if $c \leq   c'$ \\
(5) & $\bR^{a',b}_{c',d} ~\subsetneq \bR^{a,b}_{c,d}$ & \qquad & if $a \leq a'$ and $ c \leq c$.
\end{tabular}
\end{center}
\end{rem}

The {\bf support} of an indecomposable $I \in \inds$ is the subset $\supp(I) \subset Q_0$ consisting of all vertices $x$ for which the vector space $I_x$ is nontrivial (or equivalently, those vertices $x$ where $\udim_I(x)$ is nonzero). Since the indecomposables of $\Rep_\nf(Q)$ can be uniquely identified by their supports, we illustrate these five containments $I' \subsetneq I$ from Remark \ref{rem:indecinc} by depicting the supports of $I'$ (shaded red) and $I$ (outlined blue). 

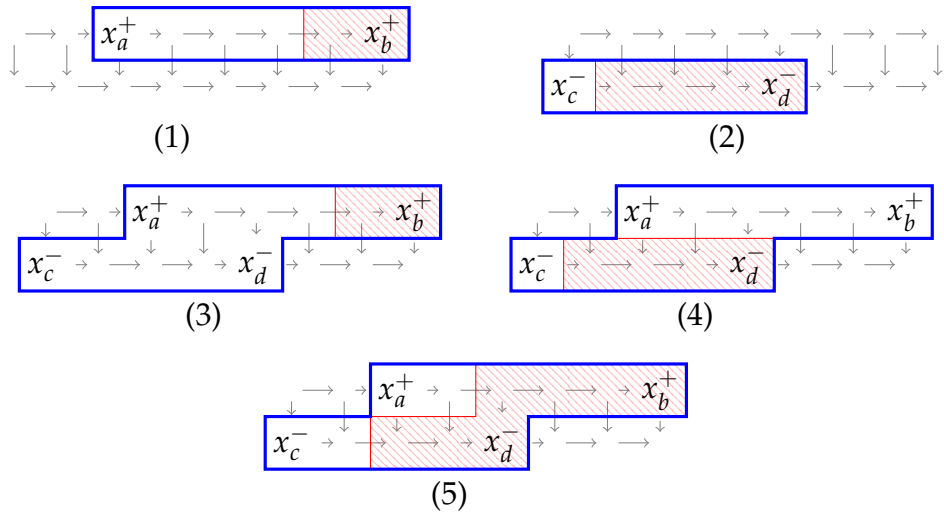
\[ \figureunderlyinggrid{0.7}{
\node (2+) at (2.5,-0.5){${x_a^+}$};
\node (7+) at (7.5,-0.5){${x_b^+}$};
}{
\node at (3.5,-2.5){(1)};
\draw[redstyle] (6,0) rectangle (8,-1);
\draw[blue, very thick ] (2,-1)-- (2,0) -- (8,0) -- (8,-1) -- (5,-1.) --cycle;
}\hspace{4em}
\figureunderlyinggrid{0.7}{
\node (0-) at (0.5,-1.5){${x_c^-}$};
\node (4-) at (4.5,-1.5){${x_d^-}$};
}{
\node at (3.5,-2.5){(2)};
\draw[redstyle ] (1,-1) rectangle (5,-2);
\draw[blue, very thick  ] (0,-2) -- (0,-1) -- (2,-1) -- (5,-1.) -- (5,-2.0) --cycle;
}
\]
\[ \figureunderlyinggrid{0.7}{
\node (2+) at (2.5,-0.5){${x_a^+}$};
\node (7+) at (7.5,-0.5){${x_b^+}$};
\node (0-) at (0.5,-1.5){${x_c^-}$};
\node (4-) at (4.5,-1.5){${x_d^-}$};
}{
\node at (3.5,-2.5){(3)};
\draw[redstyle] (6,0) rectangle (8,-1);
\draw[blue, very thick,  fill opacity=0.1] (0,-2) -- (0,-1) -- (2,-1)-- (2,0) -- (8,0) -- (8,-1) -- (5,-1.) -- (5,-2.0) --cycle;
}\qquad
\figureunderlyinggrid{0.7}{
\node (2+) at (2.5,-0.5){${x_a^+}$};
\node (7+) at (7.5,-0.5){${x_b^+}$};
\node (0-) at (0.5,-1.5){${x_c^-}$};
\node (4-) at (4.5,-1.5){${x_d^-}$};
}{
\node at (3.5,-2.5){(4)};
\draw[redstyle ] (1,-1) rectangle (5,-2);
\draw[blue, very thick,  ] (0,-2) -- (0,-1) -- (2,-1)-- (2,0) -- (8,0) -- (8,-1) -- (5,-1.) -- (5,-2.0) --cycle;
}\]
\[
\figureunderlyinggrid{0.7}{
\node (2+) at (2.5,-0.5){${x_a^+}$};
\node (7+) at (7.5,-0.5){${x_b^+}$};
\node (0-) at (0.5,-1.5){${x_c^-}$};
\node (4-) at (4.5,-1.5){${x_d^-}$};;
}{
\node at (3.5,-2.5){(5)};
\draw[redstyle ] (2,-2) -- (2,-1) -- (4,-1)-- (4,0) -- (8,0) -- (8,-1) -- (5,-1.) -- (5,-2.0) --cycle;
\draw[blue, very thick,  fill opacity=0.1] (0,-2) -- (0,-1) -- (2,-1)-- (2,0) -- (8,0) -- (8,-1) -- (5,-1.) -- (5,-2.0) --cycle;
}
\]
Let $\leq$ be the flow partial order on $\VS$ from Definition \ref{def:floworder}. For each indecomposable $I \in \inds$, we let $\min(I)$ denote the set of minimal vertices lying in the subposet $(\supp(I),\leq)$.

\begin{definition}\label{def:indpart}
For every integer $k \geq 0$  and subset $S\subset Q_0$, define the set
\[ \inds_S^k:=\set{I\in \mathcal \inds \mid \sum_{x \in \VS} \dim I_x = k \text{ and } \min(I)=S}.\]
\end{definition}
There is a partition $\inds = \coprod_{k,S}\inds^k_S$ where $k$ ranges over $\bbZ_{>0}$ while $S$ ranges over subsets of $\VS$. The constituent part $\inds_S^k$ is empty unless $S$ has one of two possible forms: \begin{itemize}
    \item $\set{x}$ for any vertex $x = x_a^+$ or $x=x_c^-$,
    \item $\set{x_a^+,x_c^-}$ with $0 \leq c < a \leq \ell$.
\end{itemize}
Here are the supports of certain $I \in \inds_S^6$ for nontrivial choices of $S\subset Q_0$:
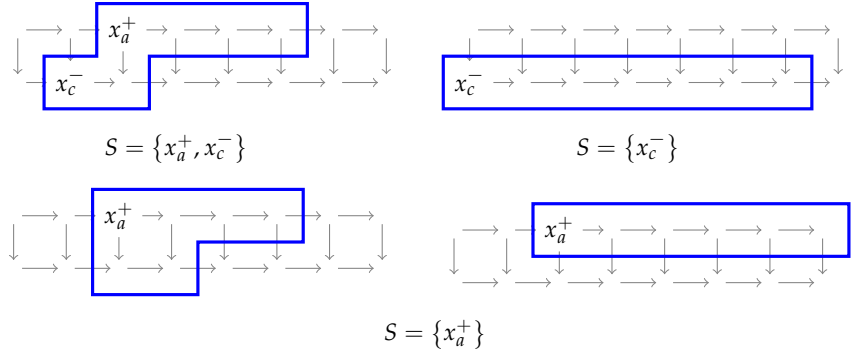
\[{\scriptsize \figureunderlyinggrid{0.7}{
\node (2+) at (2.5,-0.5){$x_a^+$};
\node (1-) at (1.5,-1.5){$x_c^-$};
}{
\node at (3.5,-2.75){$S=\set{x_a^+,x_c^-}$};
\draw[blue, very thick,  fill opacity=0.1] (1,-2) -- (1,-1) -- (2,-1)-- (2,0) -- (6,0) -- (6,-1) -- (3,-1.) -- (3,-2.0) --cycle;
}\qquad
\figureunderlyinggrid{0.7}{
\node (0-) at (0.5,-1.5){$x_c^-$};
}{
\node at (3.5,-2.75){$S=\set{x_c^-}$};
\draw[blue, very thick,  fill opacity=0.1] (0,-2) -- (0,-1) -- (2,-1) -- (7,-1.) -- (7,-2.0) --cycle;
}}
\]
\[{\scriptsize 
\figureunderlyinggrid{0.7}{
\node (2+) at (2.5,-0.5){$x_a^+$};
}{
\node at (8.5,-2.75){$S=\set{x_a^+}$};
\draw[blue, very thick,  fill opacity=0.1] (2,-2) -- (2,0) -- (6,0) -- (6,-1) -- (4,-1.) -- (4,-2.0) --cycle;
}\hspace{-2.em}
\figureunderlyinggrid{0.7}{
\node (2+) at (2.5,-0.5){$x_a^+$};
}{
\node at (3.5,-2.75){};
\draw[blue, very thick,  fill opacity=0.1] (2,-1)-- (2,0) -- (8,0) -- (8,-1) -- (5,-1.) --cycle;
}}
\]

\begin{lemma} \label{lemma:ladderlinearindep}
Given any integer $k>0$ and subset $S \subset \VS$ of vertices, the set of dimension vectors
$\set{\udim_I \mid I\in\inds_{S}^k}$ is linearly independent in the vector space of maps $\VS \to \bbR$.
\end{lemma}

\begin{proof}
It suffices to consider the nontrivial cases where $|\mathcal \inds_S^k|>1$, which can only occur when either $S = \{x_i^+\}$ for some $i$ or when $S = \{x_i^+,x_j^-\}$ for some $i > j$. We claim that every indecomposable $\bR^{a,b}_{c,d} \in \inds^k_S$ is uniquely determined by $b$. To verify this claim, note first that if $S = \{x_i^+\}$ for some $i \in \set{0,1\dots,\ell}$ then we must have $a = i$ and $c \in \set{i,\infty}$, with $c = \infty$ occurring if and only if $b = i+k-1$. Otherwise, if $S=\{x_i^+,x_j^-\}$ for $i>j$, then $a=i$ and $c=j$. In both cases, $d$ is determined by the formula
\[
k = (b-a)+(d-c)+2.
\]
Thus, the map $\iota:\inds_S^k \to [\ell]_\infty$ sending each $\bR^{a,b}_{c,d}$ to $b$ is injective, as claimed above. Now consider an $\bbR$-linear combination of the form
\[
v:=\sum_{I\in\mathcal \inds_S^k} r_I\udim_I.
\] 
A brief examination of the supports of indecomposables lying in $\inds^k_S$ reveals that the value $v(x^+_j)$ depends only on those $I$ which satisfy $\iota(I) \geq j$. Thus, the real numbers $r_I$ may be recovered by descending induction on $\iota(I)$. 
\end{proof}

Let us fix a collection of irrational numbers $\set{t_{p,q}}$ indexed by pairs of integers $0 \leq q < p$ so that the following inequalities hold: 
\begin{align}\label{eq:t}
\frac{q}{p-q} < t_{p,q} < \frac{q+1}{p -(q+1)}.
\end{align}

\begin{definition}\label{def:ladderparams}
For each integer $k> 0$ and subset $S\subset Q_0$ for which $\inds_S^k$ is nonempty, define the central charge $\alpha^k_S:\VS \to \bbR$ as 
\[
\alpha_S^k:=\begin{cases}\delta_{x} &\text{if }S=\{x\}. \\
\delta_{x_a^+}+ t_{k,a-c} \cdot \delta_{x_c^-}&\text{if }S=\{x_a^+,x_c^-\}.
\end{cases}
\]
Here $\delta_x$ is the skyscraper central charge at vertex $x$ (from Definition \ref{def:skyscraper-weight}) while the $t_{\bullet,\bullet}$ are irrational numbers chosen to satisfy \eqref{eq:t}.
\end{definition} 

We now compute the HN type of every indecomposable $I \in \inds$ along these central charges. The calculation below makes essential use of spanning subrepresentations $\subrep{I_S}$, which were introduced in Definition \ref{def:spanrep}.

\begin{lemma}
    \label{lemma:specialladder} Fix any $k> 0$ and $S\subset Q_0$ for which $\inds_S^k$ is nonempty, and define
    \[
\lambda_S^k:=\begin{cases}1/k&\text{if }|S|=1,\\
(1+t_{k,a-c})/k&\text{if }S=\set {x_a^+,x_c^-}.
\end{cases}
\]
The following assertions are equivalent for every indecomposable $I \in \inds^{k'}_{S'}$:
\begin{enumerate}
    \item The $\alpha_S^k$-slope of $I$ equals $\lambda_S^k$. 
    \item Both $k = k'$ and $S \subset S'$ hold.
\end{enumerate} 
If either assertion is true, then we also have that $I$ is $\alpha^k_S$-semistable if and only if $S = S'$.
\end{lemma} 
\begin{proof} 
We abbreviate $\alpha^k_S$ to $\alpha$ and $\lambda^k_S$ to $\lambda$. Recall from Definition \ref{def:indpart} that either $S=\set x$ or $S = \set{x_a^+,x_c^-}$ with $a > c$. In the first case, the desired equivalence follows from computing
\begin{align}\label{eq:slope1}
\mu_{\alpha}(I)=\begin{cases}1/{k'}&\text{if }x\in \supp(I) \\ 0&\text{otherwise}\end{cases}.
\end{align}
In the case $S=\set {x_a^+,x_c^-}$ with $a > c$, we similarly have:
\begin{align}\label{eq:slope2}
\mu_{\alpha}(I)=\begin{cases}{1}/{k'}&\text{if }S \cap \supp (I)=\set {x_a^+},\\
t_{k,a-c}/{k'}&\text{if } S \cap  \supp(I)=\set {x_c^-},\\
(1+t_{k,a-c})/k'&\text{if }S \subset \supp(I),\\
0&\text{otherwise.}
\end{cases}
\end{align}
Since $t_{k,a-c}$ is irrational by assumption, $\mu_{\alpha}(I)$ determines $S \cap \supp(I)$, as desired.

We now assume that $I$ is $\alpha$-semistable (in addition to satisfying $\mu_\alpha(I)=\lambda$), and seek to show that $S = S'$. Since $\subrep{I_S}$ is a subrepresentation of $I$ with $\mu_\alpha(\subrep{I_S}) \geq \lambda$, we must have $I=\subrep{I_S}$. This forces $I$ to lie in $\inds_S^k$, thus ensuring $S' = S$ as desired. Conversely, if $S=S'$, then $I$ lies in $\inds_S^k$. By Lemma \ref{lemma:sumstability}, it suffices to show that indecomposable  subrepresentations  $I' \subsetneq I$ have smaller $\alpha$-slope than $I$. It is readily checked that any subrepresentation $I' \subset I$ must be both equalised and nestfree, so it suffices to consider only those $I'$ which have been listed in Remark \ref{rem:indecinc}. Of these, the $I'$ with nonzero $\alpha$-slope all have the form $\subrep{I_{S'}}$ for $S'\subsetneq S$. When $|S|=1$, there are no such $I'$ to consider; and when $S=\set{x_a^+,x_c^-}$ for $a > c$, the only relevant $I'$ are $\subrep{I_{x_a^+}}$ and $\subrep{I_{x_c^-}}$. Thus, the desired semistability of $I$ reduces to verifying two inequalities:
\[
\frac{1}{|\set{x \geqslant x_a^+} \cap \supp(I)|} \leq \frac{1+t_{k,a-c}}{k}  \geq \frac{t_{k,a-c}}{|\set{x \geqslant x_c^-} \cap \supp(I)|}.
\]
Both hold by the constraints imposed on $t_{k,a-c}$ in \eqref{eq:t}. 
\end{proof}

We recall from \eqref{rem:hntypenotation} that given a central charge $\alpha:\VS \to \bbR$, the HN type $\HNtype{\alpha}{V}$ of any $V \in \Rep_\nf(Q)$ may be viewed as a function $\bbR \to \bbN^\VS$. The following result describes the values of the function associated to each central charge $\alpha^k_S$ from Definition \ref{def:ladderparams} at the corresponding slope $\lambda^k_S$ from Lemma \ref{lemma:specialladder}.

\begin{prop} \label{prop:ladderhntypes}Consider any $(k,S)$ such that $\inds_S^k$ is nonempty, and let $I\in \inds$. Then for $\alpha := \alpha^k_S$ and $\lambda := \lambda^k_S$, we have   \begin{align}\label{eq:computeHNladdergood}
        \HNtype{\alpha}{I}(\lambda)=\begin{cases}
            \udim_{\subrep{I_S}}&\text{if } \subrep{I_S}\in\inds_S^k\\
            0&\text{otherwise}.
        \end{cases}
    \end{align}
\end{prop}

\begin{proof}Letting $I^\bullet$ be the HN filtration $\bHN^\bullet_\alpha(I)$, we consider three possible cases:

{{\bf Case 1}}: $I \in \inds^k_S$. In this case, $I=\subrep{I_S}$ and we know by Lemma \ref{lemma:specialladder} that $I$ is $\alpha$-semistable of $\alpha$-slope $\lambda$. Hence $I^\bullet$ is trivial and each side of \eqref{eq:computeHNladdergood} is equal to  $\udim_I$.

 {{\bf Case 2}}: $I \notin \inds^k_S$ and $S \not\subset \supp(I)$
 . In particular, $\subrep{I_S}$ does not lie in $\inds_S^k$ so that the right-hand side of \eqref{eq:computeHNladdergood} is $0$. Here we have $S\not\subset \supp (I^i/I^{i-1})$ for all $i$.  If $S=\set x$, then $S \cap\textrm{supp}(I)= \emptyset$ so that from Definition \ref{def:ladderparams}, $\mu_\alpha(I^i/I^{i-1})=0 \neq \lambda$ for any step $i$ of the filtration, whence $\HNtype{\alpha}{I}(\lambda) = 0$. On the other hand, if $S=\set{x_a^+,x_c^-}$, then $S \cap\textrm{supp}(I)$ can be $\emptyset$, $\set{x_a^+}$ or $\set{x_c^-}$. Then from Definition \ref{def:ladderparams}, the slopes $\mu_\alpha(I^i/I^{i-1})$  lie in either $\bbQ$ or in $t_{k,a-c} \cdot \bbQ$, neither of which contain $\lambda = (1+t_{k,a-c})/k$, whence $\HNtype{\alpha}{I}(\lambda) = 0$. 

 {{\bf Case 3}: $I\notin \inds_S^k$ and $S \subset \supp(I)$.} In this case, we claim that $I^\bullet$ is the one-step filtration $0\subsetneq \subrep{I_S}\subsetneq I$. By uniqueness of HN filtrations, it suffices to show that $\subrep{I_S}$ and $I/\subrep{I_S}$ are $\alpha$-semistable, with $\mu_\alpha(\subrep{I_S})>\mu_\alpha(I/\subrep{I_S})$. We know from Lemma \ref{lemma:specialladder} that $\subrep{I_S}$ is $\alpha$-semistable with strictly positive slope. And since  $S\cap \supp(I/\subrep{I_S})$ is empty by constructions, all the subrepresentations of $I/\subrep{I_S}$ have $\alpha$-slope $0$. Thus, we obtain the desired result: $\HNtype{\alpha}{V}$ equals $\udim_{\subrep{I_S}}$ whenever $\subrep{I_S}$ lies in $\inds^k_S$, and is trivial otherwise.
\end{proof}

To complete the proof of Theorem \ref{thm:mainladder}, recall the collection of central charges $A=\set{\alpha_S^k}$ from Definition \ref{def:ladderparams}, and consider some $V \in \Rep_\nf(Q)$ with decomposition $V \simeq \bigoplus_{I\in \inds}I^{r_I}$. We show that the multiplicities $(r_I)_{I\in\inds}$ can be recovered from $\HNtype AV$ by descending strong induction on the partial order $\subset$ from Remark \ref{rem:indecinc}. To this end, note that for indecomposables $I\in \inds_S^k$ and $J \in \inds$, we have $\subrep{J_S}=I$ if and only if $I$ is a subrepresentation of $J$. By combining Propositions \ref{prop:slopesum} and \ref{prop:ladderhntypes}, we get
\[
\HNtype{\alpha^k_S}{V}(\lambda^k_S)=\sum_{I\in \mathcal \inds_{S}^k}r^+_I \cdot \udim_{I},
\]
where $r^+_I := \sum_{J}r_{J}$ for $J$ ranging over representations in $\inds$ which satisfy $J \supseteq I$. By  Lemma \ref{lemma:ladderlinearindep}, the integer $r^+_I$ can be obtained from $\HNtype A V$ for each $I\in\inds_S^k$. Finally, assume by the inductive hypothesis that $r_J$ is known for all $J\supsetneq I$ in $\inds$.  Now  $r_I$ can be recovered from $\HNtype A V$  as the difference $r_I=r^+_I-\sum_{ J\supsetneq I}r_J$.

\begin{rem}
It follows from Definition \ref{def:ladderparams} that the set of complete central charges $A$ has cardinality 
\[
c(\ell) := \frac{2\ell^3 + 3\ell^2 + 13 \ell + 12}{6},
\] and hence grows cubically with the length $\ell$ of the underlying ladder quiver. Moreover, one can choose different irrational numbers $t_{p,q}$ in \eqref{eq:t} to generate uncountably many other such complete sets of the same cardinality $c(\ell)$. It is not clear to us (for large $\ell$) whether {\em all} complete sets of central charges for nestfree representations of ladder quivers can be obtained in this fashion; nor is it clear whether $c(\ell)$ is the minimal size for such a complete set.
\end{rem}

\bibliographystyle{abbrv}
\bibliography{refs}

\begin{thebibliography}{10}

\bibitem{alon2008probabilistic}
N.~Alon and J.~H. Spencer.
\newblock {\em The Probabilistic Method}.
\newblock John Wiley \& Sons, 3 edition, 2008.

\bibitem{atiyah-bott}
M.~F. Atiyah and R.~Bott.
\newblock {The Yang–Mills equations over Riemann surfaces}.
\newblock {\em Philosophical Transactions of the Royal Society London Series
  A}, 308(1505):523--615, 1982.

\bibitem{bollobas}
B.~Bollob\'as.
\newblock {\em Modern Graph Theory}.
\newblock Number 184 in Graduate Texts in Mathematics. Springer, 1998.

\bibitem{botnan2022rectangle}
M.~B. Botnan, V.~Lebovici, and S.~Oudot.
\newblock On rectangle-decomposable 2-parameter persistence modules.
\newblock {\em Discrete \& Computational Geometry}, pages 1--24, 2022.

\bibitem{botnanoudot}
M.~B. Botnan, S.~Oppermann, and S.~Oudot.
\newblock Signed barcodes for multi-parameter persistence via rank
  decompositions.
\newblock In {\em 38th International Symposium on Computational Geometry (SoCG
  2022)}, volume 224 of {\em Leibniz International Proceedings in Informatics
  (LIPIcs)}, pages 19:1--19:18. Schloss Dagstuhl -- Leibniz-Zentrum f{\"u}r
  Informatik, 2022.

\bibitem{bridgeland2007stability}
T.~Bridgeland.
\newblock Stability conditions on triangulated categories.
\newblock {\em Annals of Mathematics}, pages 317--345, 2007.

\bibitem{wildladder}
M.~Buchet and E.~Escolar.
\newblock Realizations of indecomposable persistence modules of arbitrarily
  large dimensions.
\newblock {\em Journal of Computational Geometry}, 13(1):298--326, 2022.

\bibitem{zigzag_pers}
G.~Carlsson and V.~de~Silva.
\newblock Zigzag persistence.
\newblock {\em Foundations of Computational Mathematics}, 10(4):367--405, 08
  2010.

\bibitem{multi}
G.~Carlsson and A.~Zomorodian.
\newblock The theory of multidimensional persistence.
\newblock {\em Discrete and Computational Geometry}, 42:71--93, 06 2007.

\bibitem{clause2022discriminating}
N.~Clause, W.~Kim, and F.~Memoli.
\newblock The discriminating power of the generalized rank invariant.
\newblock {\em arXiv:2207.11591}, 2022.

\bibitem{curry2014sheaves}
J.~M. Curry.
\newblock {\em Sheaves, cosheaves and applications}.
\newblock University of Pennsylvania, 2014.

\bibitem{escolar2016persistence}
E.~G. Escolar and Y.~Hiraoka.
\newblock Persistence modules on commutative ladders of finite type.
\newblock {\em Discrete \& Computational Geometry}, 55:100--157, 2016.

\bibitem{hnzigzag}
M.~Fersztand, V.~Nanda, and U.~Tillmann.
\newblock {H}arder-{N}arasimhan filtrations and zigzag persistence.
\newblock {\em arXiv:2211.07553}, 2022.

\bibitem{gabriel}
P.~Gabriel.
\newblock Unzerlegbare {D}arstellungen {I}.
\newblock {\em Manuscripta mathematica}, 6:71--104, 1972.

\bibitem{haiden2020semistability}
F.~Haiden, L.~Katzarkov, M.~Kontsevich, and P.~Pandit.
\newblock Semistability, modular lattices, and iterated logarithms.
\newblock {\em Journal of Differential Geometry}, forthcoming.

\bibitem{harada}
M.~Harada and G.~Wilkin.
\newblock Morse theory of the moment map for representations of quivers.
\newblock {\em Geometriae Dedicata}, 150:307--353, 2011.

\bibitem{hnfilt}
G.~Harder and M.~S. Narasimhan.
\newblock On the cohomology groups of moduli spaces of vector bundles on
  curves.
\newblock {\em Math. Ann}, 212:215–248, 1975.

\bibitem{host}
H.~Harrington, N.~Otter, H.~Schenck, and U.~Tillmann.
\newblock Stratifying multiparameter persistent homology.
\newblock {\em SIAM Journal on Applied Algebra and Geometry}, 3(3):439--471,
  2019.

\bibitem{saecular}
G.~Henselman-Petrusek and R.~Ghrist.
\newblock Saecular persistence.
\newblock {\em arXiv:2112.04927 [math.CT]}, 2021.

\bibitem{huylehn}
D.~Huybrechts and M.~Lehn.
\newblock {\em The geometry of moduli spaces of sheaves}.
\newblock Cambridge University Press, 2010.

\bibitem{barcodebases}
E.~Jacquard, V.~Nanda, and U.~Tillmann.
\newblock The space of barcode bases for persistence modules.
\newblock {\em Journal of Applied and Computational Topology}, 7:1--30, 2023.

\bibitem{ks}
M.~Kashiwara and P.~Schapira.
\newblock Persistent homology and microlocal sheaf theory.
\newblock {\em Journal of Applied and Computational Topology}, 2:83--113, 2018.

\bibitem{king}
A.~D. King.
\newblock Moduli of representations of finite dimensional algebras.
\newblock {\em The Quarterly Journal of Mathematics}, 45(4):515--530, 1994.

\bibitem{totalstability}
R.~Kinser.
\newblock Total stability functions for type $\mathbb{A}$ quivers.
\newblock {\em Algebras and Representation Theory}, 25:835--845, 2022.

\bibitem{kirillov_quiver}
A.~{Kirillov Jr}.
\newblock {\em Quiver representations and quiver varieties}, volume 174 of {\em
  Graduate Studies in Mathematics}.
\newblock American Mathematical Society, 2016.

\bibitem{bettis}
K.~P. Knudson.
\newblock A refinement of multi-dimensional persistence.
\newblock {\em Homology, Homotopy and Applications}, 10(1):259--281, 2008.

\bibitem{lesnickthesis}
M.~Lesnick.
\newblock The theory of the interleaving distance on multidimensional
  persistence modules.
\newblock {\em Foundations of Computational Mathematics}, 15:614--650, 2015.

\bibitem{rivet}
M.~Lesnick and M.~Wright.
\newblock Interactive visualization of 2-d persistence modules.
\newblock {\em arXiv:1512.00180}, 2015.

\bibitem{macpat}
R.~MacPherson and A.~Patel.
\newblock Persistent local systems.
\newblock {\em Advances in Mathematics}, 386:107795, 2021.

\bibitem{patel}
A.~McCleary and A.~Patel.
\newblock Edit distance and persistence diagrams over lattices.
\newblock {\em SIAM Journal on Applied Algebra and Geometry}, 6(2):134--155,
  2022.

\bibitem{miller}
E.~Miller.
\newblock Homological algebra of modules over posets.
\newblock {\em arXiv:2008.00063 [math.AT]}, 2020.

\bibitem{mumfordicm}
D.~Mumford.
\newblock Projective invariants of projective structures and applications.
\newblock {\em Proceedings of the International Congress of Mathematicians
  (Stockholm, 1962)}, pages 526--530, 1963.

\bibitem{mumford}
D.~Mumford, J.~Fogarty, and F.~Kirwan.
\newblock {\em Geometric Invariant Theory}.
\newblock Springer, 1994.

\bibitem{oudot2015persistence}
S.~Oudot.
\newblock {\em Persistence theory: from quiver representations to data
  analysis}, volume 209 of {\em Mathematical Surveys and Monographs}.
\newblock The American Mathematical Society, 2015.

\bibitem{moanswer}
V.~Phan.
\newblock Posets with cardinality bounds on upward-closed subsets.
\newblock MathOverflow.
\newblock \url{https://mathoverflow.net/q/442239} (version: 2023-03-07).

\bibitem{Reineke_2003}
M.~Reineke.
\newblock The {H}arder-{N}arasimhan system in quantum groups and cohomology of
  quiver moduli.
\newblock {\em Inventiones Mathematicae}, 152(2):349--368, 2003.

\bibitem{rudakov}
A.~Rudakov.
\newblock Stability for an abelian category.
\newblock {\em Journal of Algebra}, 197(1):231--245, 1997.

\bibitem{schiffler2014quiver}
R.~Schiffler.
\newblock {\em Quiver Representations}.
\newblock CMS Books in Mathematics. Springer International Publishing, 2014.

\end{thebibliography}
\end{document}